\documentclass[12pt]{article}
\usepackage[margin=2cm]{geometry}
\usepackage{amsfonts}
\usepackage{amssymb}
\usepackage{amsthm}
\usepackage{amsmath}
\usepackage{latexsym}
\usepackage{enumitem}
\usepackage{color}
\usepackage{mathrsfs}
\begin{document}
\allowdisplaybreaks[4]
\newtheorem{lemma}{Lemma}
\newtheorem{pron}{Proposition}
\newtheorem{thm}{Theorem}
\newtheorem{Corol}{Corollary}
\newtheorem{exam}{Example}
\newtheorem{defin}{Definition}
\newtheorem{remark}{Remark}
\newtheorem{property}{Property}
\newcommand{\la}{\frac{1}{\lambda}}
\newcommand{\sectemul}{\arabic{section}}
\renewcommand{\theequation}{\sectemul.\arabic{equation}}
\renewcommand{\thepron}{\sectemul.\arabic{pron}}
\renewcommand{\thelemma}{\sectemul.\arabic{lemma}}
\renewcommand{\thethm}{\sectemul.\arabic{thm}}
\renewcommand{\theCorol}{\sectemul.\arabic{Corol}}
\renewcommand{\theexam}{\sectemul.\arabic{exam}}
\renewcommand{\thedefin}{\sectemul.\arabic{defin}}
\renewcommand{\theremark}{\sectemul.\arabic{remark}}
\renewcommand{\theproperty}{\sectemul.\arabic{property}}
\def\REF#1{\par\hangindent\parindent\indent\llap{#1\enspace}\ignorespaces}

\title{\Large 
{Asymptotics of randomly weighted sums without moment conditions of random weights\thanks{The research was supported by the National Natural Science Foundation of China (No.12371471), and the Natural Science Foundation of Jiangsu Province of China (No.BK20241905).}}}
\author{\small Qingwu Gao$^{1}$, Dimitrios G. Konstantinides$^{2}$, Charalampos  D. Passalidis$^{2}$, Yuebao Wang$^{3}$\thanks{Corresponding author. Telephone: +86 512 67422726. Fax: +86 512 65112637. E-mail:  ybwang@suda.edu.cn}, and Hui Xu$^{4}$\\
{\footnotesize\it 1. School of Mathematics, Nanjing Audit University, Nanjing, China}\\
{\footnotesize\it 2. Department of Statistics and Actuarial-Financial Mathematics, University of the Aegean, Samos, Greece}\\
{\footnotesize\it 3. School of Mathematics, Soochow University, Suzhou, China}\\
{\footnotesize \it 4. School of Statistics and Data Science, Shanghai University of Finance and Economics, Shanghai, China}\\}
\date{}

\maketitle {\noindent\small {\bf Abstract.}} \small In the paper, we investigate the asymptotics of randomly weighted sums with upper tail asymptotically independent and quasi-upper tail asymptotically independent primary random variables without requiring moment assumptions on random weights. For the case of primary random variables with regularly varying tails, we obtain more explicit results via an extension of Breiman's theorem. Then an application of the obtained results is established to asymptotically estimate for the finite-time and infinite-time ruin probabilities in a discrete-time risk model.
\\

\noindent{\small{\bf Keywords:} Asymptotics of randomly weighted sums; Upper tail asymptotic independence; Regularly varying tail distribution; Extended Breiman's theorem; Discrete-time risk model; Finite-time and infinite-time ruin probabilities;}
\\

\noindent{{\bf MSC(2020) subject classifications:}\quad 60G70, 62E20, 62P05}
\\

\section{\normalsize\bf Introduction}
\setcounter{equation}{0}\setcounter{thm}{0}\setcounter{lemma}{0}\setcounter{remark}{0}\setcounter{pron}{0}\setcounter{Corol}{0}

\qquad In this paper, we investigate the tail asymptotics of randomly weighted sums, 
in which the primary random variables (r.v.s) satisfy a fairly broad dependence structure. As an application in risk theory, we establish the asymptotic estimate of the finite-time and infinite-time ruin probabilities in a discrete-time risk model. In this field, previous results usually require certain moment conditions on the random weights. The key goal of this paper is to eliminate this restriction. To this end, we explore new analytical tools, the first of which is to establish a uniform asymptotics of weighted sums, with particular attention to extending the range of the uniform asymptotics. In addition, most of existing results are based on the tail asymptotic dependence structure among the primary r.v.s, see Subsection 1.1. The second goal of this paper is to expand the scope of the dependence structure.

The whole study of this paper is closely associated with dependence structures and distribution classes of r.v.s. Therefore, we start with some necessary concepts and properties.

For the sake of brevity, in the following, let $n$ be a positive integer without special statement. If $n=\infty$, then $1\le i\le n$ means $1\le i<\infty$ or $i\ge1$.

\subsection{\normalsize\bf Dependence structures}

\qquad We say that r.v.s $\xi_{i},\ 1\le i\le n,$ are upper tail asymptotically independent (UTAI), if $P(\xi_{i}>x)>0$ for all $x\in(-\infty,\infty)$ and all $1\le i\le n$, and
\begin{equation}\label{equation101}
\lim_{\min\{x_{i},x_{j}\}\rightarrow\infty}P(\xi_{i}>x_{i}|\xi_{j}>x_{j})=0\qquad\text{for all}\quad 1\leq i\neq j\le n;
\end{equation}
see, e.g., Sibuya \cite{S1960}. If condition (\ref{equation101}) is replaced by
\begin{equation*}\label{equation102}
\lim_{\min\{x_i,x_j\}\to\infty}P(|\xi_{i}|>x_{i}|\xi_{j}>x_{j})=0\ \ \ \ \ \ \text{for all}\ \ \ \ 1\leq i\neq j\le n,
\end{equation*}
then $\xi_{i},\ 1\le i\le n,$ are said to be tail asymptotically independent (TAI); see, e.g., Maulik and Resnick \cite{MR2004}, and  Geluk and Tang \cite{GT2009}. If condition (\ref{equation101}) is changed to
\begin{equation*}\label{equation103}
\lim_{\min\{x_{i},x_{j}\}\rightarrow\infty}P(\xi_{i}>x_{i},\ \xi_{j}>x_{j})\big(P(\xi_{i}>x_{i})+P(\xi_{j}>x_{j})\big)^{-1}=0\qquad\text{for all}\quad 1\leq i\neq j\le n,
\end{equation*}
then $\xi_{i},\ 1\le i\le n,$ are said to be quasi upper tail asymptotically independent (QUTAI); see Cheng \cite{C2014}.

Obviously, the UTAI structure includes the TAI one, and is included by the QUTAI one. And the inclusion relationships are strict, see Example \ref{exam401} below, and Proposition 4.3 of Cheng \cite{C2014}. As was noted in the literature, the class of UTAI r.v.s is fairly wide (see, e.g., Proposition \ref{pron102} below), but it has rather limited structural properties. Some basic properties of UTAI structure are summarized below.

For a r.v $\xi$, we write $\xi^+=\xi\textbf{1}_{\{\xi\ge0\}}$ and $\xi^-=-\xi\textbf{1}_{\{\xi<0\}}$.

\begin{pron}\label{pron101}
(1) R.v.s $\xi_i,\ 1\le i\le n,$ are TAI $\Rightarrow$ they are UTAI; but not otherwise.

(2) If $\xi_i,\ 1\le i\le n,$ are nonnegative r.v.s, then they are TAI $\Leftrightarrow$ they are UTAI.

(3) R.v.s $\xi_i,\ 1\le i\le n,$ are UTAI $\Leftrightarrow$ r.v.s $\xi^+_i,\ 1\le i\le n,$ are UTAI.

\end{pron}

Wang et al. \cite{WWG2013} introduced a substructure of the UTAI structure, characterized by more properties than the original structure. R.v.s $\xi_{i},\ 1\le i\le n,$ are said to be widely upper orthant dependent (WUOD), if there exists a finite positive real number $g_{U}(i)$ such that
\begin{eqnarray}\label{equation103}
&P\big(\bigcap^{i}_{j=1}\{\xi_j > x_j\}\big)\leq g_{U}(i)\prod_{j=1}^i P(\xi_j> x_j)\ \ \ \ \ \ \ \ \text{for all}\ \ \ \ x_j\in(-\infty,\infty),\ 1\leq j\leq i,
\end{eqnarray}
If (\ref{equation103}) is replaced by
\begin{eqnarray}\label{equation104}
&P\big(\bigcap^{i}_{j=1}\{\xi_j \leq x_j\}\big)\leq g_{L}(i)\prod_{j=1}^n P(\xi_j \leq x_j)\ \ \ \ \ \ \ \ \text{for all}\ \ \ \ x_j\in(-\infty,\infty),\ 1\leq j\leq i,
\end{eqnarray}
where $g_{L}(i),\ 1\le i\le n$, are finite positive real numbers, then $\xi_i,\ 1\le i\le n,$ are said to be widely lower orthant dependent (WLOD).
If $\xi_i,\ 1\le i\le n,$ are both WUOD and WLOD, they are called widely orthant dependent (WOD), where $g_U(i)$, $g_L(i)$, $1\le i\le n,$ are called the dominant coefficients of dependence structure.

Especially, when $g_L(i)=g_U(i)=1$ for all $1\le i\le n$ in (\ref{equation103}) and (\ref{equation104}), the r.v.s $\xi_i,\ 1\le i\le n,$ are called negatively upper orthant dependent (NUOD) and negatively lower orthant dependent (NLOD), respectively. If both hold, then they are said to be negatively orthant dependent (NOD); see, e.g., Ebrahimi and Ghosh \cite{EG1981}. Furthermore, the corresponding examples in Section 3 of Wang et al. \cite{WWG2013} and Example 3.1 of Liu et al. \cite{LGW2012} illustrate the relationship among 
these structures.

\begin{pron}\label{pron102}
QUTAI $\supset$ UTAI $\supset$ WUOD $\supset$ NUOD, and these inclusions are all proper.
\end{pron}

\subsection{\normalsize\bf Distribution classes}

\qquad Throughout this paper, all limits are taken as $x\rightarrow\infty$ unless stated otherwise.
For two positive functions $a(\cdot)$ and $b(\cdot)$, define $C=\limsup a(x)b^{-1}(x).$
We write $a(x)\lesssim b(x)$ or $b(x)\gtrsim a(x)$ if $C\leq1$, write $a(x)\sim b(x)$ if $a(x)\lesssim b(x)$ and $b(x)\lesssim a(x)$, write $a(x)=O\big(b(x)\big)$ if $C<\infty$, write $a(x)\asymp b(x)$ if $a(x)=O\big(b(x))$ and $b(x)=O\big(a(x))$, and write $a(x)=o(b(x))$ if $C=0$.

Let $V$ be a distribution supported on $(-\infty, \infty)$ \big(including $[a,\infty)$ or $(a,\infty)$ for some constant $a\in(-\infty,\infty)$\big) with tail distribution $\overline{V}=1-V$. For each $y>0$, define
\begin{equation*}
\overline{V}^{\ast}(y)=\limsup\overline{V}(xy)\overline{V}^{-1}(x)\ \ \ \ \text{and}\ \ \ \ \overline{V}_{\ast}(y)=\liminf\overline{V}(xy)\overline{V}^{-1}(x).
\end{equation*}
The moment index of $V$ is defined by
\begin{eqnarray*}
&\mathbf{I}_V=\sup\big\{a:\int_0^\infty y^aV(dy)<\infty\big\},
\end{eqnarray*}
and the upper and lower Matuszewska indices of distribution $V$ are given, respectively, by
\begin{eqnarray*}
\mathbf{J}_V^+=-\lim\limits_{y\to\infty}\big(\ln\overline{V}_*(y)\big)\ln^{-1}y\ \ \ \ \text{and}\ \ \ \
\mathbf{J}_V^-=-\lim\limits_{y\to\infty}\big(\ln \overline{V}^*(y)\big)\ln^{-1}y.
\end{eqnarray*}

We say that a distribution $V$ on $(-\infty, \infty)$ belongs to the class of dominatedly varying tailed distributions, denoted by $V\in\mathcal{D}$, if $\overline{V}^{\ast}(y)<\infty$ for all (or, equivalently, for some) $0<y<1$; and belongs to the class of consistently varying tailed distributions, denoted by $V\in\mathcal{C}$, if $\lim_{y\uparrow1}\overline{V}^{\ast}(y)=1$. Moreover, a distribution $V$ on $(-\infty, \infty)$ is said to belong to the class of regularly varying tailed distributions, denoted by $V\in\mathcal{R}_{-\alpha}$ for some $\alpha\ge0$, if for each $y>0$, $\overline{V}^{\ast}(y)=\overline{V}_{\ast}(y)=y^{-\alpha}.$
At this time,
$$\overline{V}(x)\sim x^{-\alpha} L(x),$$
where the function $L(\cdot)$ on $[0, \infty)$ belongs to the slowly varying function class $\mathcal{R}^0_{0}$. And the class $\mathcal{R}^0_{0}$ is a subclass of the function class $\mathcal{R}^0_\gamma$ for some $\gamma\in(-\infty,\infty)$ defined by
$$\mathcal{R}^0_\gamma=\big\{f(\cdot)\ \text{on}\ [0,\infty):\ f(x)>0\ \ \text{for all}\ \ x\ge0\ \ \ \text{and}\ \ \ f(tx)\sim t^{\gamma}f(x)\ \ \ \text{for each}\ \ t>0\big\}.$$
Clearly, if $V\in\mathcal{R}_{-\alpha}$, then $\overline{V}(\cdot)\in\mathcal{R}^0_{-\alpha}$.

For a distribution $V\in\mathcal{D}$, the proposition below is a combination of results in Feller \cite{F1971}, Bingham et al. \cite{BGT1987}, Cline and Samorodnitsky \cite{CS1994}, Embrechts et al. $\cite{EKM1997}$, 
Tang and Tsitsiashvili $\cite{TT2003}$ and Zou et al. \cite{ZWW2012}. For example, see Proposition 3.1 of Zou et al. \cite{ZWW2012} for Proposition \ref{pron103}(1).

\begin{pron}\label{pron103} (1) $V\in\mathscr{D}$ if and only if for each distribution $U$ on $(-\infty,\infty)$ satisfying $\overline{U}(x)=o(\overline{V}(x))$,
there exists a positive function $g(\cdot)$ on $[0,\infty)$ such that
\begin{equation*}
g(x)\downarrow0,\ \ \ \ \ \ \ xg(x)\uparrow\infty\ \ \ \ \ \ \ \text{and}\ \ \ \ \ \ \ \overline{U}\big(xg(x)\big)=o\big(\overline{V}(x)\big).
\end{equation*}

(2) If $V\in\mathcal{D}$, then $0\le \mathbf{J}_V^-\le \mathbf{I}_V\le \mathbf{J}_V^+<\infty$.

(3) If $V\in\mathcal{D}$, then $\overline{V}(x)=o(x^{-p})$ for each $p<\mathbf{J}_V^-$ and $x^{-p}=o\big(\overline{V}(x)\big)$ for each $p>\mathbf{J}_V^+$.

(4) If $V\in\mathcal{D}$, then for each $p>\mathbf{J}_V^+$, there exist two positive constants $C_1=C_1(V,p)$ and $C_2=C_2(V,p)$ such that
$$\overline{V}(y)\overline{V}^{-1}(x)\le C_1(xy^{-1})^{p}\ \ \ \ \ \ \ \ \ \ \ \text{for each pair}\ \ x\ge y\ge C_2.$$
\end{pron}

We say that a distribution $V$ on $(-\infty, \infty)$ belongs the class of long-tailed distributions, denoted by $V\in\cal{L}$, if for all (or, equivalently, for some) fixed $y\in(-\infty, \infty)$,
\begin{equation*}
\overline{V}(x+y)\sim\overline{V}(x).
\end{equation*}
Furthermore, a distribution $V$ on {$[0, \infty)$ is said to belong to the class of subexponential distributions, denoted by $V\in\cal{S}$}, if
\begin{equation*}
\overline{V^{*2}}(x)\sim2\overline{V}(x),
\end{equation*}
where $V^{*2}$ denotes the convolution of $V$ with itself. If $V$ has support on $(-\infty,\infty)$,
then $V\in\mathcal{S}$ if $V_+\in\mathcal{S}$, where $V_+(\cdot)=V(\cdot)\textbf{1}_{[0,\infty)}(\cdot)$. The class $\mathcal{S}$ was introduced by Chistyakov \cite{C1964}; and for comprehensive treatments, see Embrechts et al. \cite{EKM1997}, Foss et al. \cite{FKZ2013}, Wang \cite{W2022}, Leipus et al. \cite{LSK2023}, Wang et al. \cite{WCX2026}, and others.

A useful property of the long-tailed distribution is their insensitivity to shifts, which follows from Corollary 2.5 of Cline and Samorodnitsky \cite{CS1994}; see also Proposition \ref{pron104}(1) below. For a systematic summary of the insensitivity, see Foss et al. \cite{FKZ2013}. Further discussion can also be found in Remark 1.2 of Zhou et al. \cite{ZWW2012}; see also Proposition \ref{pron104}(2).

Define a set of functions as
\begin{equation}\label{equation105}
\mathcal{H}_V=\big\{h(\cdot):0<h(x)\uparrow\infty,\ h(x)x^{-1}\downarrow0\ \text{and}\ \overline{V}(x+y)\sim\overline{V}(x)\ \text{holds uniformly for}\ |y|\le h(x)\big\}.
\end{equation}
If $h(x)x^{-1}\downarrow0$ in (\ref{equation105}) is replaced by $h(x)x^{-1}\to0$, then the corresponding set is denoted by $\mathcal{\widehat{H}}_V$.

\begin{pron}\label{pron104}
(1) A distribution $V\in\cal{L}$ if and only if the set $\mathcal{H}_V$ is not empty.

(2) If a function $h(\cdot)\in\mathcal{H}_V$ for a distribution $V\in\mathcal{L}$, then $ch(\cdot)\in\mathcal{H}_V$ for each $c>0$.

(3) Under the condition of (2), there exists a positive function $f(\cdot)$ on $[0,\infty)$ satisfying $f(x)\uparrow\infty$ such that function $h_1(\cdot)=f(\cdot)h(\cdot)\in\mathcal{\widehat{H}}_V$.
\end{pron}

\proof We only need to prove (3). For each integer $n\ge1$, Proposition \ref{pron104}(2) implies that both $nh(\cdot)$ and $(n+1)h(\cdot)$ belong to $\mathcal{H}_V$. Hence there exists $x_n=x_n\big(V,h(\cdot)\big)>0$ large enough such that, when $x\ge x_n$,
\begin{eqnarray*}
&1-n^{-1}<\overline{V}\big(x-nh(x)\big)\overline{V}^{-1}(x)\le\overline{V}\big(x-(n+1)h(x)\big)\overline{V}^{-1}(x)<1+(n+1)^{-1}<1+n^{-1}.
\end{eqnarray*}
Set
$$x_{n+1}>2x_n,\ \ \ n\ge1,\ \ \ \ \ \ \ \text{and}\ \ \ \ \ \ \ nh(x_n)x_n^{-1}\to0\ \ \ \text{as}\ \ \ n\to\infty;$$
and further define a function $f(\cdot)$ as
\begin{eqnarray*}
&f(x)=\textbf{1}_{[0,x_1)}(x)+\sum_{n=1}^\infty\Big(\frac{x-x_n}{x_{n+1}-x_n}+n\Big)\textbf{1}_{[x_n,x_{n+1})}(x),\ \ \ \ \ x\ge0.
\end{eqnarray*}
Let $h_1(\cdot)=f(\cdot)h(\cdot)$. Clearly, $h_1(x)\uparrow\infty$, and when $x\in[x_n,x_{n+1})$, since $f(x)\leq n+1$ and $h(x)/x$ is decreasing, then
$$f(x)h(x)x^{-1}\le(n+1)h(x_n)x_n^{-1}\to0,\ \ \ \ \ \ \ \ \text{as}\ \ \ n\to\infty,$$
and
\begin{eqnarray*}
&1-n^{-1}<\overline{V}\big(x-h_1(x)\big)\overline{V}^{-1}(x)\le\overline{V}\big(x-(n+1)h(x)\big)\overline{V}^{-1}(x)<1+(n+1)^{-1}<1+n^{-1},
&\end{eqnarray*}
which yields that $h_1(\cdot)\in\mathcal{\widehat{H}}_V$.
$\hspace{\fill}\Box$\\

For the above distribution classes, the following inclusion relationships are proper, namely,
\begin{eqnarray*}
&\mathcal{R}=\bigcup_{\alpha\ge0}\mathcal{R}_{-\alpha}\subset\mathcal{C}\subset\mathcal{L}\cap\mathcal{D}\subset\mathcal{S}\subset\mathcal{L}.
\end{eqnarray*}
However, the classes $\mathcal{L}$ and $\mathcal{D}$ are not mutually inclusive; see, for example, the famous Peter-Paul distribution in Goldie \cite{G1978}.

In this paper, our main focus is on the intersection class $\mathcal{L}\cap\mathcal{D}$ and its subset $\mathcal{R}$.

\subsection{\normalsize\bf Main results}

\qquad In this subsection, we study the asymptotics of randomly weighted sums of UTAI and QUTAI primary r.v.s, which is the main result of this paper and has important applications in risk theory, see Section 3. And further discussions on their conditions and conclusions are delayed to the Appendix in Section 4.
For related works, we refer to Resnick and Willekens \cite{RW1991}, Tang and Tsitsiashvili \cite{TT20032}, Wang and Tang \cite{WT2006}, Zhang et al. \cite{ZSW2009}, Chen and Yuen \cite{CY2009}, Yi et al. \cite{YCS2011}, Hazra and Maulik \cite{HM2012}, Olvera-Cravioto \cite{OC2012}, Li \cite{L2013}, Tang and Yuan \cite{TY2014},
Cheng \cite{C2014}, Cheng and Cheng \cite{CC2018}, Li \cite{L2018}, Yang et al. \cite{YCY2024}, Wang et al. \cite{WMG2026}, and so on. As mentioned earlier, most existing results impose moment conditions on the random weights and dependence structures among the primary r.v.s. We now revisit Theorem 2.3 of Li \cite{L2013}.

\noindent\textbf{Theorem 1.A.} Let $X_{i},\ 1\le i\leq n,$ be TAI r.v.s with  distribution $F_i$ belonging to $\mathcal{L}\cap\mathcal{D}$ on $(-\infty, \infty),\ 1\le i\leq n,$ and let $W_{i},\ 1\le i\leq n,$ be positive r.v.s independent of $X_{i},\ 1\le i\leq n$. If $EW^p_i<\infty$ for some $p>\max\{\mathbf{J}_{F_i}^+,\ 1\le i\le n\},\ 1\le i\le n$, then
\begin{eqnarray}\label{101100}
&P\big(\sum_{i=1}^{n}W_{i}X_{i}>x\big)\sim\sum_{i=1}^{n}P(W_{i}X_{i}>x).
\end{eqnarray}

Henceforth, we respectively call $X_i$, $W_i$, and $W_iX_i$, $1\le i\le n$, the primary r.v.s, the random weights, and the increments of the sum $\sum_{i=1}^{n}W_{i}X_{i}$. Accordingly, we call $\sum_{i=1}^{n}W_{i}X_{i}$ the randomly weighted sum. Let $G_i$ and $H_i$ be the distributions of $W_i$ and $Z_i=X_iW_{i}$, respectively, $1\le i\le n$. Write
$$\mathbf{J}^+_0=\max\{\mathbf{J}_{F_i}^+,\ 1\le i\le n\}.$$
In Theorem 1.A, the random weights satisfy the conditions described as follows:\\
Case $1^*$: $EW^p_i<\infty$ for some $p>\mathbf{J}^+_0,\ 1\le i\le n$.\\
And the complementary case is denoted as Case $2^*$. Hence, Theorem 1.A naturally leads to the following two questions.

\vspace{0.2cm} \noindent\textbf{Question 1}. Can we remove the moment conditions on the random weights? 
In other words, what results can we get in Case $2^*$? 

\vspace{0.2cm} \noindent\textbf{Question 2}. Can the asymptotic relation (\ref{101100}) be extended to the case when the primary r.v.s follow the UTAI, or even QUTAI, structure, and their random weights require no moment conditions?

For Question 1, positive results have been obtained under additional conditions different from those in Theorem 1.A, and Theorems 1.1, 1.2 and Corollary 1.1 of Cheng \cite{C2014}. We here only recall one result of the latter involving QUTAI r.v.s.

\noindent\textbf{Theorem 1.B.} Let $X_{i},\ 1\le i\le n,$ be QUTAI r.v.s with distribution $F_i$ on $(-\infty, \infty),\ 1\le i\le n$, and let $W_{i},\ 1\le i\le n,$ be positive r.v.s, independent of $X_{i}, \ 1\le i\le n$. If $F_i\in\mathcal{C},\ 1\le i\le n$, and for any $u>0$,
\begin{equation}\label{1071}
\overline{G_i}(ux)=o\big(\overline{H_{i}}(x)\big),\ \ \ \ \ \ \ \ \ 1\le i\le n,
\end{equation}
then relation (\ref{101100}) holds.

Since the class $\mathcal{C}$ is properly included in $\mathcal{L}\cap\mathcal{D}$, we first establish the asymptotic relation (\ref{101100}) with the distributions of the primary r.v.s belonging to $\mathcal{L}\cap\mathcal{D}$.


Before we present our main results, the significance of this discussion is briefly explained below.

In several areas of stochastic processes, such as stochastic recurrence equations or in linear processes, the behavior of the series
\begin{eqnarray} \label{115}
\sum_{i=1}^{\infty}E\max\{W_i^{\alpha+\delta},W_i^{\alpha-\delta}\}\ \ \ \ \ \ \text{for some}\ \ \ \ \delta>0,
\end{eqnarray}
under the assumption $F \in \mathcal{R}_{-\alpha}$, is important to determine the memory of the process (if the series of (\ref{115}) converges, then we have short memory, while if it diverges we have long memory). It is sometimes mistakenly believed that the long range dependence `breaks' the single big jump principle. However, this is not generally true, as shown in Theorem \ref{thm102} above, which also includes the situation that the sums of the finite number of terms, in (\ref{115}) diverge. In light of the single big jump in comparison with the dependence structures, the conjecture from Ko and Tang \cite{KT2008} that `the single big jump remains valid when the dependence structures are not positive enough' holds when the dependence concerns the primary r.v.s $X_i,\ i\ge1$, rather than the random weights $W_i,\ i\ge1$. An example where the single big jump `fails' because of the strong dependence among the primary r.v.s can be found in Mikosch and Samorodnitsky \cite{MS2000}.

We now give the first result of this paper, which answers Questions 1 and 2 with a different idea from those in the existing results, and can serve as a complement to Theorems 1.A and 1.B.

\begin{thm}\label{thm102}
(1) Let $X_i,\ 1\le i\le n,$ be TAI r.v.s with distributions $F_i$ belonging to $\cal{L}\cap\cal{D}$ on $(-\infty, \infty),\ 1\le i\le n$, and let $W_{i},\ 1\le i\le n,$ be r.v.s, and independent of $X_{i},\ 1\le i\le n$. For each function $h(\cdot)\in\prod_{i=1}^n\mathcal{H}_{F_i}$, there exists a positive function $f_1(\cdot)=f_1\big(\cdot;h(\cdot)\big)$ on $(0,\infty)$ such that
\begin{eqnarray}\label{10711}
&1\ge f_1(x)\downarrow0\ \ \ \ \ \ \ \ \text{and}\ \ \ \ \ \ \ \ h_1(\cdot)=h(\cdot)f^{-1}_1(\cdot)\in\prod_{i=1}^n\mathcal{\widehat{H}}_{F_i}.
\end{eqnarray}
If there exists a $h(\cdot)\in\prod_{i=1}^n\mathcal{H}_{F_i}$ satisfying
\begin{eqnarray}\label{110001}
&\overline{G_i}\big(h(x)\big)=o\big(\sum_{j=1}^n\overline{H_{j}}(x)\big),\ \ \ \ \ \ \ \ \ 1\le i\le n,
\end{eqnarray}
then there exists a positive function $f_2(\cdot)=f_2\big(\cdot;h(\cdot)\big)$ on $(0,\infty)$ such that
\begin{eqnarray}\label{110001a}
&1\le f_2(x)\uparrow\infty,\ \ \ \ \ \ \ \ f_2(x)h^{-1}(x)\downarrow0,
\end{eqnarray}
and
\begin{eqnarray}\label{11000}
&\overline{G_i}\big(f_2(x)\big)=o\big(\sum_{j=1}^n\overline{H_{j}}(x)\big),\ \ \ \ \ \ \ 1\le i\le n.
\end{eqnarray}
{For the above $h(\cdot)$, further if
\begin{equation}\label{110000}
G_i\big(f_1(x)\big)=o\big(h^{-p}(x)\big),\ \ \ \ \ \ \ \ \ 1\le i\le n,
\end{equation}
for some constant $p>\mathbf{J}_0^+$, then relation (\ref{101100}) holds.}

(2) In (1), let $X_{i}$, $1\le i\le n,$ be UTAI r.v.s. Then relation (\ref{101100}) still holds, if the above functions $h(\cdot)$ and $f_2(\cdot)$ satisfy the condition that, when $n\ge2$,
\begin{eqnarray}\label{10800a}
&F_i\big(-h(x)f_2^{-1}(x)(n-1)^{-1}\big)=o\big(\sum_{j=1}^n\overline{F_j}(x)\big),\ \ \ \ \ \ \ \ \ 1\le i\le n.
\end{eqnarray}
\end{thm}

\begin{remark}\label{rem102}
(1) In Theorem \ref{thm102}(2), condition (\ref{10800a}) is necessary in a certain sense; see Example \ref{exam402} below. This condition requires that the left tail of distribution $F_i$ is light enough for each $1\le i\le n$. And we note that condition (\ref{10800a}) is unnecessary if $F_i$ is supported on $[a_i,\infty)$ for some $a_i\in(-\infty,\infty),\ 1\le i\le n$. In addition, function $h_1(\cdot)=h(\cdot)f_2^{-1}(\cdot)(n-1)^{-1}$ also belongs to $\bigcap_{i=1}^n\mathcal{H}_{F_i}$ clearly.

(2) For $1\le i\le n$, conditions (\ref{110001}) and (\ref{110000}) require that the left tail and the right tail of the distribution $G_i$ are light enough, respectively. In particular, condition (\ref{110000}) is unnecessary if $G_i$ is supported on $[a_i,\infty)$ for some $a_i>0,\ 1\le i\le n$. At this time, the only condition of this theorem is (\ref{110001}).

(3) In general, there are many functions and distributions that satisfy all the conditions of Theorem \ref{thm102}. Therefore, the conclusion (\ref{101100}) holds. See Example \ref{exam403} below, in which there exists a $W_i$ such that $EW^p_i=\infty$ for each $p>\mathbf{J}_0^+$. On the contrary, in Case $1^*$, by Proposition \ref{pron103}(3) and Markov's inequality,
\begin{eqnarray*}
&\overline{G_i}(x)=o\big(\overline{F_i}(x)\big)=o\big(\overline{H_i}(x)\big)=o\big(\sum_{j=1}^i\overline{H_j}(x)\big),\ \ \ \ \ 1\le i\le n.
\end{eqnarray*}

(4) When $F_i\in\mathcal{C},\ 1\le i\le n$, by Proposition \ref{pron103}(1), we can find a function $h(\cdot)\in\prod_{i=1}^n\mathcal{H}_{F_i}$ such that  condition (\ref{110001}) is equivalent to the condition that
\begin{eqnarray}\label{11500}
&\overline{G_i}(x)=o\big(\sum_{j=1}^n\overline{H_j}(x)\big),\ \ \ \ \ 1\le i\le n.
\end{eqnarray}
Generally, for each distribution $F_i,\ 1\le i\le n$, condition (\ref{11500}) is weaker than the condition that
\begin{eqnarray}\label{11600}
&\overline{G_i}(x)=o\big(\overline{H_i}(x)\big),\ \ \ \ \ 1\le i\le n.
\end{eqnarray}
See Remark \ref{rem201}(3) below.

In the case that $F_i\in\mathcal{D},\ 1\le i\le n$, condition (\ref{11600}) is equivalent to condition (\ref{1071}), appearing in Theorem 1.B.
\end{remark}

Based on Theorem \ref{thm102}, a more interesting question was raised.

\vspace{0.2cm} \noindent\textbf{Question 3}. For each $1\le i\le n$, can we decompose $\overline{H_i}(x)$ into the product of the tail distribution of the primary r.v. $X_i$ and a constant related to the random weight $W_i$? In this way, the conclusion of Theorem \ref{thm102} are more clear, intuitive and easy to apply.

For the convenience of answering Question 3, we all set $F_i=F$, and set $W_i=\prod_{j=1}^iY_j$ with distribution $G_i$, $1\leq i\leq n$, where $Y_j,\ 1\leq i\leq n$, are independent r.v.s with a common distribution $G=G_1$ on $(0,\infty)$. Further, we consider distribution $F\in\cal{R}_{-\alpha}$ for some $\alpha\ge0$ corresponding to a function $L(\cdot)\in\mathcal{R}^0_0$.
If condition (\ref{1140}) below holds, then by Proposition \ref{pron103}(1), there exists a positive function $f(\cdot)$ on $[0,\infty)$ such that
\begin{equation}\label{1121}
f(x)\uparrow\infty,\ \ \ \ \ \ \ x^{-1}f(x)\downarrow0\ \ \ \ \ \ \ \text{and}\ \ \ \ \ \ \ \overline{G}\big(f(x)\big)=o\big(\overline{H_1}(x)\big).
\end{equation}
Further, we consider the following three cases.

Case 1. $EY_1^{p}<\infty$ for some $p>\alpha\ge0$.

Case 2. $EY_1^{\alpha}<\infty$, but $EY_1^{p}=\infty$ for each $p>\alpha\ge0$ and 
\begin{eqnarray}\label{1131}
\lim\sup_{1\leq y\leq f(x)}L(xy^{-1})L^{-1}(x)<\infty.
\end{eqnarray}

Case 3. $EY_1^{\alpha}=\infty$ and (\ref{1131}) is satisfied, 
where, clearly, $\alpha>0$.

The above classification comes from Yang and Wang \cite{YW2013}, which is motivated by Denisov and Zwart \cite{DZ2007}. For further study based on the classification, see, for example, Cui and Wang \cite{CW2025}.

As a consequence, we present the second result of this paper to answer Question 3 positively. Given its more concise conditions and sharper conclusion, we highlight the result as the main theorem of this paper. Since $\mathcal{R}\subset\mathcal{C}$, we can invoke Theorem 1.B.

In the following, we denote the function $EY_1^{\alpha}\textbf{\emph{\emph{1}}}_{\{Y_1\le f(\cdot)\}}$ on $[0,\infty)$ by $I(\cdot)=I\big(\cdot;G,f(\cdot),\alpha\big)$.

\begin{thm}\label{thm101}
Let $X_i,\ 1\le i\le n,$ be QUTAI r.v.s with common distribution $F\in\cal{R}_{-\alpha}$ on $(-\infty, \infty)$ for some $\alpha\ge0$ corresponding to a function $L(\cdot)\in\mathcal{R}^0_0$, and let $Y_i,\ 1\le i\le n,$ be independent r.v.s with a common distribution $G$ on $(0,\infty)$ and independent of $X_i,\ 1\le i\le n$. If 
\begin{eqnarray}\label{1140}
&\overline{G}(x)=o\big(\overline{H_1}(x)\big),
\end{eqnarray}
then there exist a function $f(\cdot)$ satisfying condition (\ref{1121}), and in one of the above Cases 1-3,
\begin{eqnarray}\label{1141}
&P\big(\sum_{i=1}^{n}X_{i}\prod_{j=1}^{i}Y_j>x\big)\sim\big(\sum_{i=1}^{n}I^i(x)\big)\overline{F}(x).
\end{eqnarray}
Specifically, in Cases 1 and 2, the distribution of $\sum_{i=1}^{n}W_{i}X_{i}$ still belongs to the class $\cal{R}_{-\alpha}$, and
\begin{eqnarray}\label{1142}
&P\big(\sum_{i=1}^{n}X_{i}\prod_{j=1}^{i}Y_j>x\big)\sim\big(\sum_{i=1}^{n}E^iY_1^\alpha\big)\overline{F}(x);
\end{eqnarray}
and in Case 3, $I(x)\uparrow\infty$ and
\begin{eqnarray}\label{1151}
&P\big(\sum_{i=1}^{n}X_{i}\prod_{j=1}^{i}Y_j>x\big)\sim I^n(x)\overline{F}(x).
\end{eqnarray}

\end{thm}

\begin{remark}\label{rem102}
(1) In Theorem \ref{thm101}, Cases 2 and 3, which both correspond to long-memory behaviors of the model, are new settings, in which the theorem does not require $EY_1^{\alpha+\delta}<\infty$ for some $\delta>0$. Even in Case 1, the theorem still gives a novel result for QUTAI case for $Y_i,\ 1\le i\le n$.

(2) Beyond the current resluts, it is possible to discuss the case when the primary r.v.s $X_i,\ 1\le i\le n$, follow non-identical distributions $F_i,\ 1\le i\le n$, albeit with more complicated details. Furthermore, we can also study the multidimensional versions of the above two theorems.

(3) In case 1, condition (\ref{1140}) is satisfied clearly.
\end{remark}

Theorem \ref{thm101} naturally motivates a new question as follows.

\vspace{0.2cm} \noindent\textbf{Question 4}. In Case 3, does the distribution of $\sum_{i=1}^{n}X_{i}\prod_{j=1}^{i}Y_j$ also belong to the class $\mathcal{R}_{-\alpha}$, similarly to those in Cases 1 and 2? In other words, does the function $I(\cdot)$ still belong to the class $\mathcal{R}^0_0$?

For Question 4, a positive answer and a negative answer are given below in Example \ref{exam4040}.

Finally in this section, we give some intuitive and concise conditions to replace condition (\ref{110000}) in the following conclusions. In this way, $f_1(\cdot)$ is a function with almost no constraints.

\begin{Corol}\label{Corol101}
{(1) Under the conditions of Theorem \ref{thm102}, but without condition (\ref{110000}), if
\begin{eqnarray}\label{3100}
&\overline{F_i}(x)=o\big(\sum_{i=1}^{n}\overline{H_i}(x)\big),\ \ \ \ \ \ \ 1\le i\le n,
\end{eqnarray}
then
\begin{eqnarray}\label{3100a}
&P\big(\sum_{i=1}^{n}W_{i}X_{i}>x\big)\gtrsim\sum_{i=1}^{n}\overline{H_i}(x).
\end{eqnarray}}
{(2) In (1), if further $F_i\in\mathcal{C},\ 1\le i\le n$, then
\begin{eqnarray}\label{3100b}
&P\big(\sum_{i=1}^{n}W_{i}X_{i}>x\big)\lesssim\sum_{i=1}^{n}\overline{H_i}(x).
\end{eqnarray}
Therefore, relation (\ref{101100}) still holds and $n^{-1}\sum_{i=1}^{n}H_i\in\mathcal{C}$.}
\end{Corol}

For each $n\ge1$ and all $1\le i\le n$, we recall that $W_i=\prod_{j=1}^iY_j$ follow the distribution $G_i$. Obviously, for some function $f_1(\cdot)$ satisfying (\ref{10711}) with some $h(\cdot)\in\mathcal{H}_F$,
\begin{eqnarray*}
&P\big(\prod_{j=1}^iY_j\le f_1(x)\big)=P\big(\prod_{j=1}^iY^{-1}_j\ge f^{-1}_1(x)\big),\ \ \ \ \ \ \ x\in[0,\infty).
\end{eqnarray*}
Denote the distribution of $\prod_{j=1}^iY^{-1}_j$ by $K_i$, $1\le i\le n$, with $K_1=K$, which, clearly, are supported on $(0,\infty)$.

Recall from Subsection 1.2 that for a distribution $V$, its moment index and upper Matuszewska index are denoted by $\mathbf{I}_V$ and $\mathbf{J}^+_V$, respectively.

\begin{pron}\label{pron303}
Let $X_i$, $1\le i\le n$, be r.v.s with distributions $F_i\in\mathcal{C}$ on $(-\infty,\infty)$, and let $Y_j,\ 1\le i\le n$, be independent r.v.s with common distribution $G$ on $(0,\infty)$. For two functions $f_1(\cdot)\in \mathcal{R}^0_{\gamma_1}$ for some $\gamma_1>0$ and $f_2(\cdot)\in\mathcal{R}^0_{-\gamma_2}$ for some $\gamma_2>0$, relation (\ref{110000}) holds if
\begin{eqnarray}\label{1231}
0<\gamma^{-1}_1\gamma_2\mathbf{J}_0^+=\gamma^{-1}_1\gamma_2\max\{\mathbf{J}_{F_i}^+,\ 1\le i\le n\}<\mathbf{I}_K\leq\infty.
\end{eqnarray}
\end{pron}

Note here that condition (\ref{1231}) is satisfied by a broad class of distribution $K$; for examples, all light-tailed distributions (see Embrechts et al. \cite{EKM1997}), all Weibull-type distributions (see Arendarczyk and D\c{e}bicki \cite{AD2011}), all lognormal distributions (see Embrechts et al. \cite{EKM1997}), and all the distributions in $\mathcal{R}_{-\delta}$ for each $\delta>\gamma^{-1}_1\gamma_2\mathbf{J}_0^+$. In the first three cases, $\mathbf{I}_K=\infty$.

The remaining part of this paper is organized as follows. Section 2 proves the main results, and Section 3 presents the applications of Theorems \ref{thm101} to risk theory. In Appendix, a series of examples are provided to further illustrate and clarify the conditions and conclusions of our main results.

\section{\normalsize\bf Proofs of main results}
\setcounter{equation}{0}\setcounter{thm}{0}\setcounter{lemma}{0}\setcounter{remark}{0}\setcounter{pron}{0}\setcounter{Corol}{0}
\setcounter{Corol}{0}

\subsection{\normalsize\bf Proof of Theorem \ref{thm102}}

\qquad To avoid using moment condition of the random weights, we adopt a new method to prove Theorem \ref{thm102}, which depends on the uniform asymptotics of weighted sums. We first recall Lemma 2.1 of Liu et al. \cite{LGW2012}, which establishes the uniform asymptotics of weighted sums of TAI r.v.s, and extends the existing results for NUOD r.v.s (see Lemma 1 of Kong and Zong \cite{KZ2008}) and WUOD r.v.s (see Lemma 2.3 of Wang et al. \cite{WWG2013}). In addition, Theorem 2.1 of Li \cite{L2013} deals with the case of non-identically distributed r.v.s.

\vspace{0.2cm} \noindent\textbf{Lemma 2.A.} 
{\it If $X_{i}, 1\leq i\leq n,$ are TAI r.v.s with common distribution $F$ on $(-\infty,\infty)$ belonging to $\mathcal{L}\cap\mathcal{D}$, then for each pair of constants ${0<a\le b<\infty}$,
\begin{eqnarray}\label{106}
\lim\sup_{(w_{1},\cdots,w_{n})\in[a,b]^{n}}\left|\frac{P\left(\sum_{i=1}^{n}w_{i}X_{i}>x\right)}{\sum_{i=1}^{n}P(w_{i}X_{i}>x)}-1\right|=0.
\end{eqnarray}}

For the asymptotics of randomly weighted sums of UTAI r.v.s, it is natural to pose the following question regarding Lemma 2.A, which plays a critical role in proving Theorem \ref{thm102}.

\vspace{0.1cm} \noindent\textbf{Question 5}. Can the range $[a,b]$ of uniform asymptotics of the weighted sums be enlarged? Can $X_i,\ 1\le i\le n,$ be non-identically distributed?

The following result provides positive answers to Question 5, and is also of its own merit.

\begin{lemma}\label{lemma201}
(1) Let $X_i, 1\le i\le n,$ be TAI r.v.s with distributions $F_i$ on $(-\infty, \infty)$ belonging to $\mathcal{L}\cap\mathcal{D},\ 1\le i\le n$. Then for each function $h(\cdot)\in\prod_{i=1}^n\mathcal{H}_{F_i}$, there exist a positive function $f_1(\cdot)=f_{1}\big(\cdot;h(\cdot)\big)$ on $[0,\infty)$ satisfying condition (\ref{10711}) and another positive function $f_2(\cdot)=f_{2}\big(\cdot;h(\cdot)\big)$ on $[0,\infty)$ satisfying condition (\ref{110001a}) such that
\begin{eqnarray}\label{107}
\lim\sup_{w_{i}\in[f_{1}(x),f_{2}(x)],1\le i\le n}\left|\frac{P\left(\sum_{i=1}^{n}w_{i}X_{i}>x\right)}{\sum_{i=1}^{n}P(w_{i}X_{i}>x)}-1\right|=0.
\end{eqnarray}

(2) In (1), let $X_{i}, i\ge1,$ be UTAI r.v.s. If the above function $h(\cdot)$ satisfies condition (\ref{10800a}) for $n\ge2$, then relation (\ref{107}) still holds. 
\end{lemma}

\proof
(1) For each function $h(\cdot)\in\prod_{i=1}^n\mathcal{H}_{F_i}$, the function $f_1(\cdot)$ satisfying condition (\ref{10711}) exists by Proposition \ref{pron104}(3), and the function $f_2(\cdot)$ satisfying condition (\ref{110001a}) exists obviously.

In order to prove (\ref{107}), we only need to show the following two relations, for $n\ge 2$,
\begin{equation}\label{2020}
\limsup\ \sup_{w_{i}\in[f_1(x),f_2(x)],1\le i\le n}\frac{P\left(\sum_{i=1}^{n}w_{i}X_{i}>x\right)}{\sum_{i=1}^{n}P(w_{i}X_{i}>x)}\leq1
\end{equation}
and
\begin{equation}\label{2030}
\liminf\ \inf_{w_{i}\in[f_1(x),f_2(x)],1\le i\le n}\frac{P\left(\sum_{i=1}^{n}w_{i}X_{i}>x\right)}{\sum_{i=1}^{n}P(w_{i}X_{i}>x)}\geq1.
\end{equation}

For (\ref{2020}), a division via function $h(\cdot)\in\prod_{i=1}^n\mathcal{H}_{F_i}$ leads to
\begin{eqnarray}\label{2040}
&&P\bigg(\sum_{i=1}^{n}w_{i}X_{i}>x\bigg)=P\bigg(\sum_{i=1}^{n}w_{i}X_{i}>x,\bigcup_{j=1}^{n}\{w_{j}X_{j}>x-h(x)\}\bigg)\nonumber\\
&&\ \ \ \ \ \ \ \ \ \ \ \ \ \ \ \ \ \ \ \ \ \ \ \ \ \ \ \ \ \ +P\bigg(\sum_{i=1}^{n}w_{i}X_{i}>x,\bigcap_{j=1}^{n}\{w_{j}X_{j}\leq x-h(x)\}\bigg)\nonumber\\
&=&P_{1}(x;n)+P_{2}(x;n),\ \ \ \ \ \ \ \ \ \ \ \ \ \ \ \ x\geq0.
\end{eqnarray}
Firstly, we deal with $P_{1}(x;n)$ for two cases that $w\in[1,f_2(x)]$ and $w\in[f_1(x),1)$. When $w\in[1,\infty)$, by $h(x)x^{-1}\downarrow0$, we have
$$w^{-1}h(x)\le h(xw^{-1}).$$
By (\ref{110001a}), thus $xf^{-1}_2(x)\uparrow\infty$, which, along with $F\in\mathcal{L}$, implies that for any $0<\varepsilon<1$, there exists a sufficiently large $x_1=x_1\big(F_i,1\le i\le n,h(\cdot),f_2(\cdot),\varepsilon\big)>0$ such that, when $x\ge x_1$, it holds uniformly for all $w\in[1,f_2(x)]$ that
\begin{eqnarray}\label{2050}
P\big(wX_i>x-h(x)\big)\le P\big(X_i>xw^{-1}-h\big(xw^{-1})\big)\le (1+\varepsilon)P(wX_i>x),\ \ \ 1\le i\le n.
\end{eqnarray}
When $w\in[f_1(x),1)$, 
since $h_1(\cdot)=f^{-1}_1(\cdot)h(\cdot)\in\prod_{i=1}^n\widehat{\mathcal{H}}_{F_i}$, then for any $\varepsilon$ as above, there exists a sufficiently large $x_2=x_2\big(F_i,1\le i\le n,h(\cdot),f_2(\cdot),f_1(\cdot),\varepsilon\big)\ge x_1$ such that, when $x\ge x_2$, it holds uniformly for all $w\in[f_1(x),1)$ and $1\le i\le n$ that
\begin{eqnarray}\label{20500}
&P\big(wX_i>x-h(x)\big)\le P\big(X_i>\frac{x}{w}-h_1(x)\big)\le P\big(X_i>\frac{x}{w}-h_1(\frac{x}{w})\big)\le(1+\varepsilon)P(wX_i>x).
\end{eqnarray}
Hence, by (\ref{2050}) and (\ref{20500}), we derive that, when $x\ge x_2$,
\begin{eqnarray*}\label{2060}
\sup_{w_i\in[f_1(x),f_2(x)],1\le i\le n}\frac{P_1(x;n)}{\sum_{i=1}^{n}P(w_{i}X_{i}>x)}
\le\sup_{w_i\in[f_1(x),f_2(x)],1\le i\le n}\frac{\sum_{i=1}^{n}P(w_{i}X_{i}>x-h(x))}{\sum_{i=1}^{n}P(w_{i}X_{i}>x)}
\leq1+\varepsilon,
\end{eqnarray*}
which yields that
\begin{eqnarray}\label{2060}
&&\limsup\sup_{w_i\in[f_1(x),f_2(x)],1\le i\le n}\frac{P_1(x;n)}{\sum_{i=1}^{n}P(w_{i}X_{i}>x)}\le1.
\end{eqnarray}
Next, we deal with $P_{2}(x;n)$. Clearly, for all $w\in(0, f_2(x)]$, by (\ref{110001a}), we have
$$w^{-1}h(x)\ge f^{-1}_2(x)h(x)\uparrow\infty.$$
Then, by the above assertion, $F\in\mathcal{L}\cap\mathcal{D}\subset\mathcal{D}$, the UTAI property of $\{X_i, i\ge1\}$, and Proposition \ref{pron103}(4) with its corresponding parameters $C_1$ and $p>J_F^+$, we conclude that for any $\varepsilon$ as above, there exists a sufficiently large $x_3=x_3(F_i,1\le i\le n,h(\cdot),
f_2(\cdot),f_1(\cdot),m,\varepsilon)\ge x_2$ such that, when $x\ge x_3$, it holds uniformly for $w_i\in[f_1(x),f_2(x)]$, $1\le i\le n$, that
\begin{eqnarray}\label{2070}
&&P_2(x;n)=P\bigg(\sum_{i=1}^{n}w_{i}X_{i}>x,\bigcap_{j=1}^{n}\{w_{j}X_{j}\leq x-h(x)\},\bigcup_{k=1}^{n}\{w_{k}X_{k}>n^{-1}x\}\bigg)\nonumber\\
&\leq&\sum_{k=1}^{n}P\bigg(\sum_{1\le i\le n,i\neq k}w_{i}X_{i}>h(x),w_{k}X_{k}>n^{-1}x\bigg)\nonumber\\
&\leq&\sum_{k=1}^{n}P\bigg(\bigcup_{1\le i\le n,i\neq k}\Big\{w_{i}X_{i}>\frac{h(x)}{n-1}\Big\},w_{k}X_{k}>n^{-1}x\bigg)\nonumber\\
&\leq&\sum_{k=1}^{n}\bigg(\sum_{1\le i\le n,i\neq k}P\Big(X_{i}>\frac{h(x)}{w_i(n-1)}\Big|X_{k}>\frac{x}{w_kn}\Big)\bigg)P\bigg(X_k>\frac{x}{w_kn}\bigg)\nonumber\\
&<&\varepsilon C_1^{-1}n^{-p}\sum_{k=1}^{n}P(w_kX_k>n^{-1}x)\nonumber\\
&<&\varepsilon\sum_{k=1}^{n}P(w_kX_k>x).
\end{eqnarray}
Thus, by (\ref{2070}) and the arbitrariness of $\varepsilon$, we obtain that
\begin{eqnarray}\label{2070000}
&&\sup_{w_i\in[f_1(x),f_2(x)],1\le i\le n}\frac{P_2(x;n)}{\sum_{i=1}^{n}P(w_iX_i>x)}=0.
\end{eqnarray}
As a result, for each integer $n\ge2$, relation ({\ref{2020}) follows by substituting (\ref{2060}) and (\ref{2070000}) into (\ref{2040}).

For (\ref{2030}), we use function $h(\cdot)$ as above to obtain that
\begin{eqnarray}\label{2080}
&&P\bigg(\sum_{i=1}^{n}w_{i}X_{i}>x\bigg)\geq P\bigg(\sum_{i=1}^{n}w_{i}X_{i}>x, \bigcup_{1\le j\le n}\{w_{j}X_{j}>x+h(x)\}\bigg)\nonumber\\
&\ge&\sum_{j=1}^{n}P\bigg(\sum_{i=1}^{n}w_{i}X_{i}>x, w_{j}X_{j}>x+h(x)\bigg)-\sum_{1\leq j<i\leq n}P\big(w_{i}X_{i}>x+h(x),w_{j}X_{j}>x+h(x)\big)\nonumber\\
&\geq&\sum_{j=1}^{n}P\bigg(\sum_{1\leq i\leq n,i\neq j}w_{i}X_{i}>-h(x), w_{j}X_{j}>x+h(x)\bigg)\nonumber\\
&&\qquad\qquad\qquad-\sum_{1\leq j<i\leq n}P\big(w_{i}X_{i}>x+h(x), w_{j}X_{j}>x+h(x)\big)\nonumber\\
&&\geq\sum_{j=1}^{n}P\big(w_{j}X_{j}>x+h(x)\big)-\sum_{j=1}^{n}P\bigg(\sum_{1\le i\le n,i\neq j}w_{i}X_{i}\le -h(x),w_{j}X_{j}>x+h(x)\bigg)\nonumber\\
&&\qquad\qquad\qquad-\sum_{1\leq j<i\leq n}P(w_{i}X_{i}>x,\ w_{j}X_{j}>x)\nonumber\\
&&=:P_{3}(x;n)-P_{4}(x;n)-P_{5}(x;n).
\end{eqnarray}
Firstly, we consider $P_3(x;n)$. By $F\in\cal{L}$ and the derivation of (\ref{2060}) for $P_1(x;n)$, for any $\varepsilon$ as above, there exists a sufficiently large $x_4=x_4(F_i,1\le i\le n,h(\cdot),f_2(\cdot),f_1(\cdot),\varepsilon)\ge x_3$ such that, when $x\ge x_4$,
\begin{eqnarray}\label{equation209}
\inf_{w_i\in[f_1(x),f_2(x)],1\le i\le n}\frac{P_3(x;n)}{\sum_{i=1}^{n}P(w_{i}X_{i}>x)}\geq1-\varepsilon.
\end{eqnarray}
Secondly, we deal with $P_{4}(x;n)$ by a similar method to that in $P_{2}(x;n)$. For any $\varepsilon$ as above and for each integer $m\ge2$, by $F\in\mathcal{L}\cap\mathcal{D}\subset\mathcal{D}$, $f_2(x)h^{-1}(x)\to0$ and the TAI property of $X_i,i\ge1$, there exists a sufficiently large
$x_5=x_5(F_i,1\le i\le n,h(\cdot),f_2(\cdot),f_1(\cdot),\varepsilon)\ge x_4$ such that, when $x\ge x_5$,
\begin{eqnarray}\label{2100}
&&\sup_{w_{i}\in[f_1(x),f_2(x)],1\le i\le n}\frac{P_{4}(x;n)}{\sum_{j=1}^{n}P({w_jX_j}>x)}\nonumber\\
&\le&\sup_{w_{i}\in[f_1(x),f_2(x)],1\le i\le n}\frac{\sum_{j=1}^{n}P\big(\bigcup_{1\leq i\leq n,i\neq j}
\big\{w_iX_{i}\le-\frac{h(x)}{n-1}\big\},w_jX_{j}>x\big)}{\sum_{j=1}^{n}P(w_jX_j>x)}\nonumber\\
&\leq&\sup_{w_{i}\in[f_1(x),f_2(x)],1\le i\le n}\frac{\sum_{j=1}^{n}\sum_{1\leq i\leq n,i\neq j}
P\big(|X_{i}|>{\frac{h(x)}{w_i(n-1)}}\big|X_{j}>\frac{x}{w_j}\big)P(X_{j}>\frac{x}{w_j})}{\sum_{j=1}^{n}P(w_jX_j>x)}\nonumber\\
&<&\varepsilon.
\end{eqnarray}
Thirdly, we turn to $P_{5}(x;n)$. By the UTAI property, 
for any $\varepsilon$ as above, there exists a sufficiently large $x_6=x_6\big(F_i,1\le i\le n,h(\cdot),f_2(\cdot),f_1(\cdot),m,\varepsilon\big)\ge x_5$ such that, when $x\ge x_6$,
\begin{eqnarray}\label{2110}
&&\sup_{w_{i}\in[f_1(x),f_2(x)],1\le i\le n}\frac{P_{5}(x;n)}{\sum_{i=1}^{n}P(w_{i}X_{i}>x)}\nonumber\\
&=&\sup_{w_{i}\in[f_1(x),f_2(x)],1\le i\le n}\frac{{\sum_{1\leq j<i\leq n}P(w_{j}X_{j}>x\mid w_{i}X_{i}>x)P(w_{i}X_{i}>x)}}{\sum_{i=1}^{n}P(w_{i}X_{i}>x)}\nonumber\\
&<&\varepsilon.
\end{eqnarray}
Therefore, by (\ref{2080})-(\ref{2110}) and the arbitrariness of $\varepsilon$, relation (\ref{2030}) holds for each integer $n\ge 2$. 

Consequently, we combine (\ref{2020}) and (\ref{2030}) to show that relation (\ref{107}) holds for each $n\ge1$.\\ 
%

(2) According to the proof of Lemma \ref{lemma201}(1), we just deal with $P_4(x;n)$ for UTAI r.v.s and $n\ge2$. To this end,
we first establish a fact that, for some $p>\max\{J^+_{F_i},1\le i\le n\}$, by (\ref{10800a}) with some $h(\cdot)\in\prod_{i=1}^n\mathcal{H}_{F_i}$ and the corresponding function $f_2(\cdot)$, there exists a function $f_{1}(\cdot)=f_{1}\big(\cdot;h(\cdot),f_2(\cdot),p\big)$ satisfying (\ref{10711}) such that
\begin{eqnarray}\label{108}
&f^{-p}_{1}(x)F_i\big(-h(x)f_2^{-1}(x)(n-1)^{-1}\big)\big(\sum_{j=1}^n\overline{F_j}(x)\big)^{-1}\to0,\ \ \ \ \ \ \ 1\le i\le n.
\end{eqnarray}
For example, we take a positive function $f_{11}(\cdot)$ on $[0,\infty)$ such that $1\ge f_{11}(x)\downarrow0$ and $h(\cdot)f^{-1}_{11}(\cdot)\in\mathcal{\widehat{H}}_F$,
and take a positive function $f_{12}(\cdot)=f_{12}(\cdot;h(\cdot),p,\theta)$ on $[0,\infty)$ for some $0<\theta<1$ defined by
\begin{eqnarray*}
&f_{12}(x)=\max_{1\le i\le n}\min\big\{1,\sup_{y\ge x}\big(F_i\big(-h(y)f_2^{-1}(y)(n-1)^{-1}\big)\big(\sum_{j=1}^n\overline{F_j}(y)\big)^{-1}\big)^{\theta p^{-1}}\big\}.
\end{eqnarray*}
Now we take function $f_1(\cdot)=\max\{f_{11}(\cdot), f_{12}(\cdot)\}$ on $[0,\infty)$, which, clearly, satisfies (\ref{10711}). In the following, we prove that function $f_1(\cdot)$ satisfies (\ref{108}). In fact, for each $1\le i\le n$, by (\ref{10800a}), we have
\begin{eqnarray*}
&f^{-p}_{1}(x)\frac{F_i\big(-\frac{h(x)}{f_2(x)(n-1)}\big)}{\sum_{j=1}^n\overline{F_j}(x)}\le f^{-p}_{12}(x)\frac{F_i\big(-\frac{h(x)}{f_2(x)(n-1)}\big)}{\sum_{j=1}^n\overline{F_j}(x)}
\le\bigg(\frac{F_i\big(-\frac{h(x)}{f_2(x)(n-1)}\big)}{\sum_{j=1}^n\overline{F_j}(x)}\bigg)^{1-\theta}\to0.
\end{eqnarray*}


For each $n\ge2$ and for the above functions $h(\cdot),\ f_2(\cdot)$ and $f_1(\cdot)$, by $F_i\in\mathcal{L}\cap\mathcal{D}\subset\mathcal{D},\ 1\le i\le n$, Proposition \ref{pron103}(4), and (\ref{108}), we find that, for $\varepsilon$ as above, there exists a sufficiently large $x'_{5}=x'_{5}(F,h(\cdot),f_2(\cdot),f_1(\cdot),\varepsilon)\ge x_5$ such that, when $x\geq x'_{5}$,
\begin{eqnarray*}
&&\sup_{w_i\in[f_1(x),f_2(x)],1\le i\le n}\frac{P_{4}(x;n)}{\sum_{j=1}^{n}{P(w_jX_j>x)}}\nonumber\\
&\le&\sup_{w_i\in[f_1(x),f_2(x)],1\le i\le n}\frac{\sum_{j=1}^{n}P\big(\bigcup_{1\leq i\leq n,i\neq j}
\big\{X_{i}\le-\frac{h(x)}{w_i(n-1)}\big\}, w_jX_{j}>x\big)}{\sum_{j=1}^{n}{P(w_jX_j>x)}}\nonumber\\
&\leq&\sup_{w_i\in[f_1(x),f_2(x)],1\le i\le n}
\frac{\sum_{j=1}^{n}\sum_{1\leq i\leq n,i\neq j}P\big(X_{i}\le-\frac{h(x)}{f_2(x)(n-1)}\big)}{\sum_{j=1}^{n}{P(w_jX_j>x)}}\nonumber\\
&\leq&\sup_{w_i\in[f_1(x),f_2(x)],1\le i\le n}\frac{nC_1f^{-p}_{1}(x)\sum_{i=1}^nP\big(X^{-}_{i}\ge\frac{h(x)}{f_2(x)(n-1)}\big)}{\sum_{i=1}^n\overline{F_i}(x)}
\frac{C_1^{-1}f^{p}_{1}(x)\sum_{i=1}^n\overline{F_i}(x)}{\sum_{i=1}^{n}{\overline{F_i}\big(f^{-1}_{1}(x)x\big)}}\nonumber\\
&<&\varepsilon,
\end{eqnarray*}
which proves that relation (\ref{107}) still holds.
$\hspace{\fill}\Box$\\

Now we prove Theorem \ref{thm102}. Firstly, we prove that there exist two functions $f_1(\cdot)$ satisfying (\ref{10711}) and $f_2(\cdot)$ satisfying (\ref{110001a}) and (\ref{11000}).

For each function $h(\cdot)\in\prod_{i=1}^n\mathcal{H}_{F_i}$, by Proposition \ref{pron104}(3), there exists a positive function $f_1(\cdot)$ on $(0,\infty)$ satisfying (\ref{10711}).

For each $1\le i\le n$, by $F_i\in\mathcal{D}$ and Theorem 3.3(ii) of Cline and Samarodnitsky \cite{CS1994}, we have $H_i\in\mathcal{D}$. 
Further, for each function $h(\cdot)\in\prod_{i=1}^n\mathcal{H}_{F_i}$ satisfying (\ref{110001}), by $H_i\in\mathcal{D},\ 1\le i\le n,$ (\ref{110001}) and Proposition \ref{pron103}(1), there exists a positive function $f_2(\cdot)$ on $(0,\infty)$ satisfying (\ref{110001a}) and (\ref{11000}).

(1) Write $B_n=\bigcap_{i=1}^n\{f_1(x)\le W_i\le f_2(x)\}$. Then by using $B_n$, we have
\begin{eqnarray}\label{2150}
&P\big(\sum_{i=1}^{n}W_{i}X_{i}>x\big)=P\big(\sum_{i=1}^{n}W_{i}X_{i}>x,B_n\big)+P\big(\sum_{i=1}^{n}W_{i}X_{i}>x,B^c_n\big)
\mathcal{}=:P_1(x)+P_2(x).
\end{eqnarray}
For $P_2(x)$, by $D_n=\bigcup_{i=1}^{n}\{W_i>f_2(x)\}$, we split it into two parts as
\begin{eqnarray*}\label{2170}
&P_2(x)=P\big(\sum_{i=1}^{n}W_{i}X_{i}>x,B^c_n,D^c_n\big)+P\big(\sum_{i=1}^{n}W_{i}X_{i}>x,B^c_n,D_n\big)=:P_{21}(x)+P_{22}(x).
\end{eqnarray*}
Firstly, we deal with $P_{22}(x)$. By (\ref{11000}), 
we find that
\begin{eqnarray}\label{2190}
&P_{22}(x)\le P(D_n)\le\sum_{i=1}^{n}\overline{G_i}\big(f_2(x)\big)
=o\big(\sum_{i=1}^{n}\overline{H_{i}}(x)\big).
\end{eqnarray}
Secondly, we turn to $P_{21}(x)$. For each $1\le i\le n$, it holds that
\begin{eqnarray}\label{2180}
P(W_iX_i>x)\ge P(W_iX_i>x, 1<W_i\le b)\ge P(X_i>x)P(1<W_i\le b),
\end{eqnarray}
where $b>1$ is large enough such that $P(1<W_i<b)>0$. Thus by Proposition \ref{pron103}(4), (\ref{2180}) and (\ref{110000}), it follows that for a sufficiently large $x$,
\begin{eqnarray}\label{2200}
&&P_{21}(x)\le\sum_{i=1}^{n}\sum_{j=1}^{n}P\big(W_{i}X_{i}>n^{-1}x,0<W_j\le f_1(x),D_n^c\big)\nonumber\\
&\le&\sum_{i=1}^{n}P\big(W_{i}X_{i}>n^{-1}x,0<W_i\le f_1(x)\big)
+\sum_{i=1}^{n}\sum_{1\le j\neq i\le n}P\big(W_{i}X_{i}>n^{-1}x,0<W_j\le f_1(x),D^c_n\big)\nonumber\\
&\le&C_1n^{p+1}\bigg(f^p_1(x)\sum_{i=1}^{n}P\big(0<W_i\le f_1(x)\big)P(X_{i}>x)+f^p_2(x)\sum_{j=1}^{n}P\big(0<W_j\le f_1(x)\big)P(X_{i}>x)\bigg)\nonumber\\
&=&o\bigg(\sum_{i=1}^{n}P(W_iX_i>x)\bigg).
\end{eqnarray}
{For $P_1(x)$, according to Lemma \ref{lemma201}(1), (\ref{110000}) and the proof of $P_{21}(x)$, we have}
\begin{eqnarray}\label{2160}
&&P_1(x)= {\int\cdots\int_{f_1(x)\le w_i\le f_2(x),1\le i\le n}}P\bigg(\sum_{i=1}^{n}w_{i}X_{i}>x\bigg)P(W_i\in dw_i,1\le i\le n)\nonumber\\
&\sim&\int\cdots\int_{f_1(x)\le w_i\le f_2(x),1\le i\le n}\sum_{i=1}^{n}P(w_{i}X_{i}>x)P(W_i\in dw_i,1\le i\le n)\nonumber\\
&\sim&\sum_{i=1}^{n}P(W_{i}X_{i}>x){-\sum_{i=1}^{n}P\big(W_{i}X_{i}>x,W_i>f_2(x)\big)-\sum_{i=1}^{n}P\big(W_{i}X_{i}>x,W_i\le f_1(x)\big)}\nonumber\\
&\sim&\sum_{i=1}^{n}P(W_{i}X_{i}>x).
\end{eqnarray}
As a consequence, we combine (\ref{2150}), (\ref{2190}), (\ref{2200}) and (\ref{2160}) to reach relation (\ref{101100}) immediately.\\

(2) Except when dealing with $P_1(x)$, we need an additional condition (\ref{10800a}), and the other proofs are the same as those in (1).


\subsection{\normalsize\bf Proof of Theorem \ref{thm101}}

\qquad According to Theorem 1.B and Remark \ref{rem102}(4), it suffices to prove (\ref{11600}), then to decompose $P(X_i\prod_{j=1}^i$ $Y_j>x),\ 1\le i\le n$. To this end, we give the following lemma, which can be regraded as an extension of Breiman's theorem. For related results, we refer to Breiman \cite{B1965}, Corollary 3.6 (iii) of Cline and Samorodnitsky \cite{CS1994}, Lemma 2.2 of Konstantindes and Mikosch \cite{KM2005}, Proposition 2.1 of Denisov and Zwart \cite{DZ2007}, Theorem 2.1 of Yang and Wang \cite{YW2013}, Theorem 2.5 of Maulik and Podder \cite{MP2016}, Theorem 2.1 of Cui and Wang \cite{CW2025}, and references therein.

\begin{lemma}\label{lem202}
Let the conditions of Theorem \ref{thm101} are valid.

(1) There exists a positive function $f(\cdot)$ on $[0,\infty)$ satisfying (\ref{1121}) such that
\begin{eqnarray}\label{21111}
&\overline{H}_i(x)=P\big(X\prod_{j=1}^iY_j>x\big)\sim\big(E^{i}Y_1^\alpha\textbf{\emph{1}}_{\{Y_1\le f(x)\}}\big)\overline{F}(x)=I^{i}(x)\overline{F}(x),\ \ \ \ \ i\ge1.
\end{eqnarray}
Especially, in Cases 1 or 2, $H_i\in\mathcal{R}_{-\alpha}$ and
\begin{eqnarray}\label{2091}
\overline{H}_i(x)\sim (E^{i}Y_1^{\alpha})\overline{F}(x).
\end{eqnarray}

(2) Relation (\ref{11600}) holds.
\end{lemma}

\begin{remark}\label{rem201}
(1) In comparison to the existing versions of Breiman's theorem, the conclusion in Case 3 of Lemma \ref{lem202} is a novel contribution to the field, as it does not require that $EY_1^\alpha<\infty$. In this case, however, we are not sure that $H_i\in\mathcal{R}_{-\alpha}$, as that in Cases 1 and 2. On the one hand, for given $G$ and $f(\cdot)$, one can derive that $I(\cdot)\in\mathcal{R}^0_0$, and then determine $H_n\in\mathcal{R}_{-\alpha}$, see Example \ref{exam4040}(1) below. On the other hand, Example \ref{exam4040}(2) provides a negative answer for Question 4.


(2) Note that relation (\ref{2091}) for Cases 1 and 2 is the special cases of (4.8) in Yang and Wang \cite{YW2013}, more generally, of (3.8) in Cui and Wang \cite{CW2025}. The results in both the literature above require the condition that
\begin{eqnarray}\label{2241}
\overline{G}(x)=o\big(\overline{F}(x)\big).
\end{eqnarray}
Clearly, condition (\ref{1140}) of Theorem \ref{thm101} is weaker than condition (\ref{2241}) in Case 3. For example, we take $G=F\in\mathcal{R}_{-\alpha}$ for some $\alpha>0$ satisfying (\ref{1131}), then (\ref{1140}) is satisfied by  $I(x)\uparrow\infty$, and thus (\ref{21111}) still holds by Lemma \ref{lem202}, but (\ref{2241}) is fail. And in Case 1 and Case 2, (\ref{1140}) is equivalent to (\ref{2241}).

(3) Condition (\ref{11500}) is weaker than condition (\ref{11600}) in Theorem \ref{thm102}(1). For example, we take $n=2$ and distributions $F_i\in\mathcal{R}_{-\alpha}$ for some $\alpha>0$ with corresponding function $L_i(\cdot)\in\mathcal{R}^0_0$ satisfying condition (\ref{1131}), $i=1,2$. Further, we take $W_1=W_2$ satisfying $EW_1^\alpha<\infty$, and $G_1=G_2=F_2$. If $\overline{F_2}(x)=o\big(\overline{F_1}(x)\big)$, then by Lemma \ref{lem202}, we have
\begin{eqnarray*}
\overline{G_1}(x)=o\big(\overline{H_1}(x)\big)\ \ \ \ \ \ \text{and}\ \ \ \ \ \ \overline{G_2}(x)=o\big(\overline{H_1}(x)+\overline{H_2}(x)\big),\ \ \ \ \ \ \text{but}\ \ \ \ \ \ \
\overline{G_2}(x)\neq o\big(\overline{H_2}(x)\big).
\end{eqnarray*}
\end{remark}

\begin{proof}[\it Proof of Lemma \ref{lem202}.] According to Remark \ref{rem201}(2), we only need to prove (\ref{21111}) for Case 3 with $\alpha>0$ by induction.

When $i=1$, by (\ref{1140}), there exists a function $f(\cdot)$ satisfying condition (\ref{1121}). On the one hand, for any $0<\varepsilon<1$, there exists a number $x_0=x_0(F,G,\varepsilon)$ large enough such that, when $x\ge x_0$,
\begin{eqnarray*}
&(1-\varepsilon)\overline{H}_1(x)\le\int_0^{1}\overline{F}(x)G(dy)+\int_1^{f(x)}\overline{F}(xy^{-1})G(dy).
\end{eqnarray*}
Further, by (\ref{1131}), we know that
\begin{eqnarray*}
&\frac{(1-\varepsilon)\overline{H}_1(x)}{I(x)\overline{F}(x)}\le\int_0^{1}\frac{G(dy)}{I(x)}+\int_1^{f(x)}\frac{\overline{F}(xy^{-1})G(dy)}{I(x)\overline{F}(x)}<\infty.
\end{eqnarray*}
Thus, by the dominant convergence theorem and the arbitrariness of $\varepsilon$,
\begin{eqnarray}\label{21311}
\overline{H}_1(x)\lesssim I(x)\overline{F}(x).
\end{eqnarray}
On the other hand, by L$\acute{e}$vy's lemma and $F\in\mathcal{R}_{-\alpha}$,
\begin{eqnarray}\label{21411}
\overline{H}_1(x)\gtrsim I(x)\overline{F}(x).
\end{eqnarray}
By combining (\ref{21311}) and (\ref{21411}), relation (\ref{21111}) holds for $i=1$.

Suppose that relation (\ref{21111}) holds for $i=k$, then when $n=k+1$, we conclude by the inductive hypothesis and the methods similar to those used when $n=1$ that
\begin{eqnarray*}\label{308}
&\overline{H}_{k+1}(x)\sim\int_1^{f(x)}\overline{H_k}(xy^{-1})G(dy)\sim I^{k+1}(x)\overline{F}(x),
\end{eqnarray*}
namely that relation (\ref{21111}) holds for $i=k+1$, and therefore relation (\ref{21111}) holds for all $i\geq1$.

(2) It suffices to prove (\ref{11600}) in Case 3. For $i=2$, by (\ref{1140}), (\ref{1131}) and Lemma \ref{lem202}(1), we have
\begin{eqnarray*}\label{308}
&\overline{G}_{2}(x)=o\big(\int_1^{f(x)}\overline{H_1}(xy^{-1})G(dy)+\overline{H_1}(x)\big)=o\big(I^2(x)\overline{F}(x)+\overline{H_1}(x)\big)
=o\big(\overline{H_2}(x)\big).
\end{eqnarray*}
Therefore, (\ref{11600}) in Case 3 is holds by induction.
\end{proof}

%
%

Accordingly, by Theorem 1.B and Lemma \ref{lem202}, we show that relations (\ref{1141})-(\ref{1151}) in Theorem \ref{thm101} hold immediately.

%

\subsection{\normalsize\bf Proof of Coroliary \ref{Corol101}}

\qquad (1) For each $a>0$, let $W_i^{(1)}=W_i\textbf{1}_{\{W_i\le a\}}$ and $W_i^{(2)}=W_i\textbf{1}_{\{W_i>a\}},\ 1\le i\le n$.
Then by Theorem \ref{thm102} without condition (\ref{110000}}), Remark \ref{rem102}(2) and condition (\ref{3100}), we have
\begin{eqnarray}\label{2261a}
&P\big(\sum_{i=1}^{n}W_{i}X_{i}>x\big)\ge P\big(\sum_{i=1}^{n}W_{i}^{(2)}X_{i}>x\big)\sim\sum_{i=1}^{n}P(W_{i}^{(2)}X_{i}>x)\sim\sum_{i=1}^{n}\overline{H_i}(x).
\end{eqnarray}

(2) Using the proof of (3.8) in Cheng \cite{C2014}, and by conditions (\ref{110001}), $F_i\in\mathcal{C},\ 1\le i\le n$, (\ref{3100}) and Theorem \ref{thm102}, we know that
\begin{eqnarray}\label{2262b}
&&P\bigg(\sum_{i=1}^{n}W_{i}X_{i}>x\bigg)\le P\bigg(\sum_{i=1}^{n}W_{i}^{(1)}X^+_{i}>\varepsilon x\bigg)
+P\bigg(\sum_{i=1}^{n}W_{i}^{(2)}X_{i}>(1-\varepsilon)x\bigg)\nonumber\\
&\lesssim&\sum_{i=1}^{n}P(X^+_{i}>a^{-1}\varepsilon x)+\sum_{i=1}^{n}\overline{F_i}^*(1-\varepsilon)\overline{H_i}(x)\nonumber\\
&\sim&\sum_{i=1}^{n}\overline{H_i}\big(x)\ \ \ \ \ \ \ \ \ \text{as}\ \ \ \varepsilon\downarrow0.
\end{eqnarray}
Combining (\ref{2261a}) and (\ref{2262b}), we obtain relation (\ref{101100}) and $n^{-1}\sum_{i=1}^{n}H_i\in\mathcal{C}$.

\subsection{\normalsize\bf Proof of Proposition \ref{pron303}}

\qquad Obviously, by (\ref{1231}), there exist two constants $\delta$ and $p$ satisfying $\gamma^{-1}_1\gamma_2\mathbf{J}_F^+<\gamma^{-1}_1\gamma_2p<\delta<\mathbf{I}_K$ such that $x^\delta\overline{K}(x)\to0$.
Therefore, by Proposition 1(ii) of Cui and Wang \cite{CW2020}, $x^\delta\overline{K_i}(x)\to0$ for all $i\ge1$, which results into relation (\ref{110000}).

\section{\normalsize\bf Applications to risk theory}
\setcounter{equation}{0}\setcounter{thm}{0}\setcounter{lemma}{0}\setcounter{remark}{0}\setcounter{pron}{0}\setcounter{Corol}{0}
\setcounter{Corol}{0}

\qquad In this section, using Theorem \ref{thm101}, we give some asymptotic estimates of finite-time and infinite-time ruin probability in a discrete-time risk model, which originated from Nyrhinen \cite{N1999, N2001}.

For each $1\le i\le n$, let the net insurance loss during period $i$ be a r.v. $X_i$ with common distribution $F$ on $(-\infty,\infty)$, and let the stochastic discount factor be a r.v. $Y_i$ with common distribution $G$ on $(0,\infty)$. Assume that $X_i,\ 1\le i\le n$, and $Y_i,\ 1\le i\le n$, are mutually independent. These two are called insurance risks and financial risks in Tang and Tsitsiashvili \cite{TT2003, TT2004}, respectively. Define
\begin{eqnarray*}
&S_n=\sum_{i=1}^{n}X_i\prod_{j=1}^{i}Y_j,
\end{eqnarray*}
which represents the stochastic discounted value of aggregate net losses up to time $n$. Hence, the finite-time ruin probability up to time $n$ and the infinite-time ruin probability are defined, respectively, by
\begin{eqnarray*}
&\psi(x;n)=P\big(\max_{1\le k \le n}S_k>x\big)\ \ \ \ \ \ \ \text{and}\ \ \ \ \ \ \ \psi(x)=P\big(\sup_{k\ge1}S_k>x\big),
\end{eqnarray*}
where $x\ge0$ denotes the insurer's initial wealth.

Increasing attention has been paid to the asymptotics of the aforementioned research objectives. We refer to Tang and Tsitsiashvili \cite{TT2003, TT2004}, Goovaerts et al. \cite{GKLTV2005}, Wang and Tang \cite{WT2006}, Zhang et al. \cite{ZSW2009}, 
Shen et al. \cite{SLZ2009}, Gao and Wang \cite{GW2010}, Chen \cite{C2011}, 
Yang and Wang \cite{YW2013}, Li and Tang \cite{LT2015}, Maulik and Podder \cite{MP2016}, Yang et al. \cite{YCY2024}, Cui and Wang \cite{CW2025}, etc. It should be noted that most existing results also rely on moment conditions for the stochastic discount factor $Y_1$ as a random weight; see, for example, Theorem 5.1 of Tang and Tsitsiashvili \cite{TT2003} as follows, which is a pioneer result in this field.

Recall $W_i=\prod_{j=1}^iY_j$ with distributed $G_i$ on $(0,\infty),\ 1\le i\le n$. Clearly, $G_1=G$.

\vspace{0.2cm} \noindent\textbf{Theorem 3.A}
Consider the above discrete-time risk model with independent $X_i,Y_i,\ 1\le i\le n$, if $F\in\mathcal{L}\cap\mathcal{D}$ and $EY_1^p<\infty$ for some $p>\mathbf{J}^+_F$, then
\begin{eqnarray}\label{301}
&\psi(x;n)\sim\sum_{i=1}^nP\big(X_i\prod_{j=1}^iY_j>x\big).
\end{eqnarray}

The subsequent results have relaxed the independence assumption among $X_i,\ 1\le i\le n,$ to some dependence structures, while imposing no dependence restrictions on $Y_i,i\ge1$. See, for example, Theorem 3.1 of Shen et al. \cite{SLZ2009}. However, the above results still require that $EY_1^p<\infty$ for some $p>\mathbf{J}^+_F$.
Motivated by these developments, we consider the following question.

\vspace{0.2cm} \noindent\textbf{Question 5}. Can the moment condition on the random weight $Y_1$ in the existing results be weakened or even removed? In other words, what conclusions can be obtained without imposing any moment assumptions on $Y_1$?

At the same time, we aim to extend Theorem 3.A and related results to the case in which the net insurance losses $X_i,\ 1\le i\le n,$ are QUTAI.

According to Theorem \ref{thm102} and the fact that
\begin{eqnarray}\label{30211}
&P\big(\sum_{i=1}^{n}X_i\prod_{j=1}^{i}Y_j>x\big)\le\psi(x;n)\le P\big(\sum_{i=1}^{n}X^+_i\prod_{j=1}^{i}Y_j>x\big),
\end{eqnarray}
we can directly achieve the above two goals.

\begin{thm}\label{thm301}
(1) Consider the above discrete-time risk model with TAI $X_i,\ 1\le i\le n.$ If $F\in\cal{L}\cap\cal{D}$, and there exists some function $h(\cdot)\in\mathcal{H}_F$ and corresponding positive function $f_j(\cdot)$ on $[0,\infty),\ j=1,2$, satisfying conditions (\ref{10711})-(\ref{110000}), then relation (\ref{301}) holds.

(2) In (1), let $X_{i},\ 1\le i\le n,$ be UTAI r.v.s. Further, if condition (\ref{10800a}) is satisfied, then relation (\ref{301}) still holds.
\end{thm}

Based on Theorems \ref{thm101}, we obtain the following result immediately.

\begin{thm}\label{thm302}
(1) Consider the above discrete-time risk model with QUTAI $X_i,\ 1\le i\le n,$ and independent $Y_i,\ 1\le i\le n$. If $F\in\cal{R}_{-\alpha}$ for some $\alpha\ge0$ and condition (\ref{1140}) is satisfied, then in Cases 1-3, it holds that
\begin{eqnarray}\label{303}
&\psi(x;n)\sim\big(\sum_{i=1}^nI^i(x)\big)\overline{F}(x).
\end{eqnarray}
In particular, for Cases 1 and 2, it holds
\begin{eqnarray}\label{302}
&\psi(x;n)\sim\big(\sum_{i=1}^nE^iY_1^\alpha\big)\overline{F}(x);
\end{eqnarray}
and for Case 3, it holds
\begin{eqnarray}\label{3056}
\psi(x;n)\sim I^n(x)\overline{F}(x).
\end{eqnarray}

(2) In (1), if $EY_1^\alpha<1$ and take $n=\infty$, then
\begin{eqnarray}\label{30200}
\psi(x)\sim (EY_1^\alpha)(1-EY_1^\alpha)^{-1}\overline{F}(x).
\end{eqnarray}
\end{thm}

Finally, it should be pointed out that the asymptotics of random weighted sums are crucial for other applications such as capital allocation, see Tang and Yuan \cite{TY2014} for more details.

\section{\normalsize\bf Appendix}
\setcounter{equation}{0}\setcounter{thm}{0}\setcounter{lemma}{0}\setcounter{remark}{0}\setcounter{pron}{0}\setcounter{Corol}{0}
\setcounter{Corol}{0}

\qquad The following four examples are found to illustrate the conditions and conclusions of our main results. We begin by showing that the UTAI structure properly contains the TAI structure.

\subsection{\normalsize\bf TAI structure and UTAI structure}

\begin{exam}\label{exam401}
(1) Let $Y_1$ and $Y_2$ be two independent r.v.s supported on $(-\infty,\infty)$. If $Y_1^+=Y_1\textbf{\emph{1}}_{\{Y_1\ge0\}}$ and $Y_{2}^-=-Y_2\textbf{\emph{1}}_{\{Y_2<0\}}$ follow a common distribution $V\in\mathcal{S}$, then Example 4.2.2 of Wang \cite{W2022} show that
$$X_1=Y_1^++Y_2^-\ \ \ \ \ \ \ \ \ \text{and}\ \ \ \ \ \ \ \ \ X_2=Y_2$$
are UTAI, but are not TAI.

(2) In the above example, however, $X_1$ is supported on $[0,\infty)$, and thus $X_1$ has a different distribution from $X_2$,
which is contrary to the condition of Theorems \ref{thm102} that $X_1$ and $X_2$ have the same distribution. As a result, we need to transform Example 4.2.2 of Wang \cite{W2022}. Let
$$X_1=Y_1\textbf{\emph{1}}_{\{Y_1\ge0\}}-Y_2\textbf{\emph{1}}_{\{Y_2<0\}}+Y_1\textbf{\emph{1}}_{\{Y_1<0,Y_2\ge0\}}=Y_1^++Y_2^--Y^-_1\textbf{\emph{1}}_{\{Y_2\ge0\}}
\ \ \ \ \ \ \text{and}\ \ \ \ \ \ X_2=Y_2.$$
where $Y_1$ and $Y_2$ are defined as (1). Obviously, $X_1$ and $X_2$ are both supported on $(-\infty,\infty)$,
$$X_1^+=Y_1^++Y_2^-,\ \ \ \ \ X_1^-=Y^-_1\textbf{\emph{1}}_{\{Y_2\ge0\}},\ \ \ \ \ X_2^+=Y_2^+\ \ \ \ \ \text{and}\ \ \ \ \ \ \ X^-_2=Y_2^-.$$
Assume that $X_1^+$ and $X^+_2$ have the same distribution, and that $X_1^-$ and $X^-_2$ also have the same distribution. Hence, we know that $X_1$ and $X_2$ have the same distribution.

For each pair of positive numbers $x_1$ and $x_2$, it holds that
\begin{eqnarray*}
P(X_1>x_1,X_2>x_2)\le P(Y_1^+>x_1)P(Y_2>x_2)\le P(X_1>x_1)P(X_2>x_2),
\end{eqnarray*}
which yields that $X_1$ and $X_2$ are UTAI. Especially, we take $x_1=x_2=x>0$ and use $V\in\mathcal{S}$ to show that
\begin{eqnarray*}
P(|X_2|>x,X_1>x)
\ge P(Y_2^->x)\sim2^{-1}P(Y_1^++Y_2^->x)=2^{-1}P(X_1>x),
\end{eqnarray*}
which indicates that $X_1$ and $X_2$ are not TAI.

(3) Let $F$ be a distribution on $(-\infty,\infty)$ belonging to class $\mathcal{L}$, and $U$ be a r.v. following a uniform distribution on $(0,1)$.
Further, set $X_1=F^{\leftharpoonup}(1-U)$ and $X_2=F^{\leftharpoonup}(U)$, where $F^{\leftharpoonup}(u)=\inf\{x,F(x)\ge u\}$ is the quantile function.
Clearly, $X_1$ and $X_2$ have same distribution $F$ on $(-\infty,\infty)$.

For both $x_1$ and $x_2$ large enough, by $1-F(x_1)<F(x_2)$, we have
\begin{eqnarray*}
P(X_1>x_1,\ X_2>x_2)
=P\big(U<1-F(x_1),U>F(x_2)\big)=0,
\end{eqnarray*}
which yields that $X_1$ and $X_2$ are UTAI, and that, when $x_1$ rises slowly enough to compare with $x_2$,
\begin{eqnarray*}
P(|X_1|>x_1|X_2>x_2)=P(X_1<-x_1|X_2>x_2)=\min\{F(-x_1)\overline{F}^{-1}(x_2),1\}\to1,
\end{eqnarray*}
which indicates that $X_1$ and $X_2$ are not TAI.

(4) Final example comes from Example 6.5 of Yu et al. \cite{YWC2015} with a slight correction, see Example \ref{exam402} below for details.
$\hspace{\fill}\Box$
\end{exam}

\subsection{\normalsize\bf On condition (\ref{10800a})}

\qquad Here, we show that condition (\ref{10800a}) is necessary for {Theorem \ref{thm102}(2) in a certain sense.

\begin{exam}\label{exam402}
Assume that $\{\xi,\xi_i,\ 1\le i\le n,\}$ and $\{\eta,\eta_n,\ 1\le i\le n,\}$ are two sequences of independent and identically distributed r.v.s, and the two sequences are mutually independent of each other, where $\xi$ and $\eta$ have symmetric distributions $F_1$ and $F_2$ on $(-\infty,\infty)$, respectively. Write
$$X_{2i-1}=\eta_{2i-1}+\xi_i\ \ \ \ \ \text{and}\ \ \ \ \ X_{2i}=\eta_{2i}-\xi_i,\ \ \ \ \ 1\le i\le n.$$
For the sake of brevity, we only consider the case of $i=1$ below. According to Example 1(ii) of Lehmann \cite{L1966}, for each pari $x_1$ and $x_2$, we have
$$P(\eta_{1}+\xi_1>x_1,\ -\eta_{2}+\xi_1>x_2)\ge P(\eta_{1}+\xi_1>x_1)P(-\eta_{2}+\xi_1>x_2).$$
Thus, by Lemma 1(ii) of Lehmann \cite{L1966}, we know that $X_{1}=\eta_{1}+\xi_1$ and $X_{2}=\eta_{2}-\xi_1$
are NUOD, thus UTAI, r.v.s with a common symmetric distribution $F=F_1*F_2$ on $(-\infty,\infty)$.

\emph{}Now, we prove that $X_1$ and $X_2$ are not TAI r.v.s. In fact, we know that the distribution of $X_1+X_2=\eta_1+\eta_2$ is $F_2^{*2}$. In the following, we might as well set $F_1\in{\cal{R}_{-\alpha}}$ for some $\alpha>0$ and
\begin{eqnarray}\label{4010}
\overline{F_2^{*2}}(x)=o\big(\overline{F_1}(x)\big),
\end{eqnarray}
then according to Lemmas 2.4(ii) and 2.5(i) of Pakes \cite{P2004}, we have $F=F_1*F_2\in\mathcal{R}_{-\alpha}$. Thus for each $h(\cdot)\in\mathcal{H}_F$,
\begin{eqnarray}\label{4020}
\overline{F_1}(x)=o\big(\overline{F_1}(h(x))\big).
\end{eqnarray}
Further, by (\ref{4010}), (\ref{4020}), and the symmetry of $F$, we have
\begin{eqnarray*}
\overline{F}(x)=\overline{F_1*F_2}(x)\sim\overline{F_1}(x)=o\big(\overline{F_1}(h(x))\big)=o\big(\overline{F}\big(h(x)\big)\big)=o\big(F(-h(x))\big)
\end{eqnarray*}
and
\begin{eqnarray*}
P(X_1+X_2>x)=\overline{F_2^{*2}}(x)=o\big(\overline{F_1}(x)\big)=o\big(\overline{F}(x)\big),
\end{eqnarray*}
where the former relation indicates that condition (\ref{10800a}) for each $h(\cdot)$ (thus condition (\ref{108}) for each $p>\max\{J^+_{F_i},1\le i\le n\}$ and each $f_1(\cdot)$) is not satisfied, and the latter relation indicates that conclusion (\ref{101100}) with $W_1=W_2=1$ not valid.
This fact show that $X_1$ and $X_2$ are not TAI r.v.s.

Note that there exist many r.v.s satisfying condition (\ref{10800a}). For example, in Example \ref{exam401}(1), since $X_1$ is supported on $[0,\infty)$, then condition (\ref{10800a}) is trivially satisfied. In Example \ref{exam401}(2) and (3), relation (\ref{10800a}) also holds provided that the left tail of $Y_1$ is light enough. For another illustrative example, see Example \ref{exam403} below.
$\hspace{\fill}\Box$
\end{exam}

\subsection{\normalsize\bf On functions $f_1(\cdot)$ and $f_2(\cdot)$}

\qquad In the following example, we first give the explicit expressions of two functions $f_1(\cdot)$ and $f_2(\cdot)$ that satisfy some basic conditions of Theorem \ref{thm102}.

\begin{exam}\label{exam403}
For simplicity, we might as well set $n=2$ and $F_1=F_2=F$. For each pair constants $0<\alpha\le\beta<\infty$, let $X$ be a r.v. with distribution $F$ defined as
$$\overline{F}(x)=2^{-1}\big((2-e^{-\beta x^4+1}\emph{\textbf{1}}_{(-\infty,-1)}(x)
+\textbf{\emph{1}}_{[-1,1]}(x)+x^{-\alpha}\ln^{\alpha+1}(e-1+x)\textbf{\emph{1}}_{(1,\infty)}(x)\big),\ \ \ \ \ x\in(-\infty,\infty).$$
Clearly, $F\in\mathcal{R}_{-\alpha}\subset\mathcal{L}\cap\mathcal{D}$ on $(-\infty,\infty)$, $\mathbf{J}^+_F=\mathbf{J}^-_F=\mathbf{I}_F=\alpha$ and $E|X|^\alpha=\infty$.

Further£¬we take
$$h(x)=\textbf{\emph{1}}_{[0,1]}(x)+x\ln^{-2^{-1}}(e-1+x)\textbf{\emph{1}}_{(1,\infty)}(x),$$
$$f_1(x)=\textbf{\emph{1}}_{[0,1]}(x)+\ln^{-2^{-2}}(e-1+x)\textbf{\emph{1}}_{(1,\infty)}(x)$$
and
$$f_2(x)=\textbf{\emph{1}}_{[0,1]}(x)+x\ln^{-1}(e-1+x)\textbf{\emph{1}}_{(1,\infty)}(x),\ \ \ \ x\in[0,\infty).$$

Clearly, for the above three functions, conditions (\ref{10711}) and (\ref{10712}) are both satisfied, and functions $h(\cdot)$ and $h_1(\cdot)=f_1^{-1}(\cdot)h(\cdot)$ are both belong to $\mathcal{H}_F$ by $F\in\mathcal{R}_{-\alpha}$.
For each $\alpha<p<\alpha+1$, by $\beta\ge\alpha$, we have
$$f_1^{-p}(x)F\big(-h(x)f_2^{-1}(x)\big)\overline{F}^{-1}(x)=O\big(x^{\alpha-\beta}\ln^{2^{-2}p-\alpha-1}x\big)\to0,$$
which means that condition (\ref{108}), thus condition (\ref{10800a}), is satisfied.

Finally, for each $\rho>\mathbf{J}_F^+=\alpha$ and $i=1,2$, we define distribution $K_i$ of r.v. $W^{-1}_i$ on $(0,\infty)$ such that
$$\overline{K_i}(x)=P(W^{-1}_i>x)=\textbf{\emph{1}}_{(-\infty,0)}(x)+s(x)\textbf{\emph{1}}_{[0,1)}(x)+e^{-\rho x^{4(3-i)}}\textbf{\emph{1}}_{[1,\infty)}(x),\ \ x\in(-\infty,\infty),$$
where $s(\cdot)$ is a linear function on $[0,1]$ satisfying $s(0)=1$ and $s(1)=e^{-\rho}$. Then for $i=1,2$, we define distribution $G_i$ of random weight $W_i$ on $(0,\infty)$ as follows:
$$\overline{G_i}(x)=\textbf{\emph{1}}_{(-\infty,0]}(x)+\big(1-\overline{K_i}(x^{-1})\big)\textbf{\emph{1}}_{(0,1]}(x)
+(1-e^{-\rho})x^{-\alpha}\ln^{2^{-3+i}}(e-1+x)\textbf{\emph{1}}_{(1,\infty)}(x)\big),\ \ x\in(-\infty,\infty).$$
Clearly, $EW_i^\alpha=\infty$. Further, we take $\rho\ge p>\mathbf{J}^+_F=\alpha$, then
$$f_2^p(x)G_i\big(f_1(x)\big)=f_2^p(x)\overline{K_i}\big(f_1^{-1}(x)\big)=O\big(x^{p-\rho}\ln^{-p}x\big)\to0,$$
and
$$\overline{G_i}\big(f_2(x)\big)\overline{H_i}^{-1}(x)=O\big(\overline{G_i}\big(f_2(x)\big)\overline{F}^{-1}(x)\big)=O\big(\ln^{2^{-3+i}-1}x\big)\to0,\ \ \ \ \ i=1,\ 2,$$
that is conditions (\ref{110000}) and (\ref{11000}) are satisfied.

Therefore, all conditions of Theorem \ref{thm102} with $n=2$ are satisfied. 
$\hspace{\fill}\Box$
\end{exam}

\subsection{\normalsize\bf On the regular variation of $I(\cdot)$}

\qquad In this subsection, the positive and negative examples are presented for the regular variability of $I(\cdot)$ in Theorem \ref{thm101}, respectively.

\begin{exam}\label{exam4040}
For each pair constants $0<\alpha<\beta<\infty$, let $X_1$ be a r.v. with distribution $F$ defined by
$$\overline{F}(x)=2^{-1}\big((2-|x|^{-\beta})\emph{\textbf{1}}_{(-\infty,-1)}(x)
+\textbf{\emph{1}}_{[-1,1]}(x)+x^{-\alpha}\textbf{\emph{1}}_{(1,\infty)}(x)\big),
\ \ \ \ \ \ x\in(-\infty,\infty).$$
Clearly, $F\in\mathcal{R}_{-\alpha}$ on $(-\infty,\infty)$, and $\mathbf{J}^+_F=\mathbf{J}^-_F=\mathbf{I}_F=\alpha$.

(1) We now give a positive example. Set $0<\alpha<1$. Let $Y_1$ be a r.v. with distribution $G$ on $(0,\infty)$ such that
\begin{eqnarray}\label{4040}
\overline{G}(x)=\textbf{\emph{1}}_{(-\infty,1]}(x)+x^{-\alpha}\ln^{-1}(e-1+x)\textbf{\emph{1}}_{(1,\infty)}(x),\ \ \ \ \ x\in(-\infty,\infty),
\end{eqnarray}
which yields that $G\in\mathcal{R}_{-\alpha}$, and $EX_1^\alpha=EY_1^\alpha=\infty$. 
Further, we take a function $f(\cdot)$ as
$$f(x)=\textbf{\emph{1}}_{[0,1]}(x)+x\ln^{-1}(e-1+x)\textbf{\emph{1}}_{(1,\infty)}(x),\ \ \ \ \ \ x\in[0,\infty).$$
Then $f(\cdot)\in\mathcal{R}_1^0$, condition (\ref{1121}), thus (\ref{1140}), is satisfied, and
\begin{eqnarray}\label{404000}
I(x)=EY_1^\alpha\textbf{\emph{1}}_{\{Y_1\le f(x)\}}\sim\alpha\ln\ln\big(e-1+f(x)\big)\sim\alpha\ln\ln x\uparrow\infty.
\end{eqnarray}
The latter implies that $I(\cdot)\in\mathcal{R}_0^0$.

Further, by (\ref{404000}) and (\ref{3056}), we have an accurate and intuitive result that
\begin{eqnarray*}\label{404001}
&\psi(x;n)=P\big(\max_{1\le k \le n}S_k>x\big)\sim I^n(x)\overline{F}(x)=\alpha^n(\ln\ln x)^nx^{-\alpha},
\end{eqnarray*}
which shows that the distribution of $\max_{1\le k \le n}S_k$ belongs to the class $\mathcal{R}_{-\alpha}$.\\

(2) Next, we give a negative example. Firstly, for some {$0<\alpha<1$}, let $G_{01}$ be a distribution satisfying
$$\overline{G_{01}}(x)=\textbf{\emph{1}}_{(-\infty,1]}(x)+{x^{-\alpha-1}}
\textbf{\emph{1}}_{(1,\infty)}(x),\ \ \ \ \ x\in(-\infty,\infty).$$
Then, $G_{01}\in\mathcal{R}_{\alpha+1}$ on $[1,\infty)$, and $\int_0^\infty \overline{G_{01}}(y)dy<\infty$.

Secondly, by $G_{01}$ and the method of Proposition 1.5.3(2) of Wang \cite{W2022}, we define a distribution $G_{02}$ such that
\begin{eqnarray}\label{409}
&\overline{G_{02}}(x)=\overline{G_{01}}(x)\textbf{\emph{1}}_{(-\infty,x_1)}(x)+
\sum\limits_{i=1}^{\infty}\left(\overline{G_{01}}(x_i)\textbf{\emph{1}}_{[x_i,y_i)}(x)+\overline{G_{01}}(x)\textbf{\emph{1}}_{[y_i,x_{i+1})}(x)\right),
\ \ \ x\in(-\infty,\infty),
\end{eqnarray}
where constants {$\{x_i,\ y_i:\ i\ge1\}$ satisfying $x_{i+1}-y_{i}\ge y_{i}-x_{i}>1$, and for some constant $a>1$},
$$\lim_{i\to\infty}\overline{G_{02}}(x_i)\overline{G_{02}}^{-1}(y_i)=a.$$
Clearly, $G_{02}$ on $[1,\infty)$, $\overline{G}_{02}(x)\asymp\overline{G_{01}}(x)$, and $\int_1^\infty \overline{G_{02}}(y)dy<\infty$, but $G_{02}\notin\mathcal{L}$.

Thirdly, let $f(\cdot)$ be chosen as above. Clearly, both function $f(\cdot)$ and its derivative $f'(\cdot)$ belong to class $\mathcal{L}_0$. Based on $f(\cdot)$ and $G_{02}$, we define a distribution $G_{03}$. For each $i\ge1$, there exists a unique pair $u_{i}$ and $v_{i}$ such that $f_2(u_{i})=y_{i}$ and $f_2(v_{i})=x_{i}$. We assert that there exists a constant $b>0$ such that $u_i-v_i>b$ for all $i\ge1$.
Otherwise, there is a subsequence $\{i_j:j\ge1\}$ of $\{i:i\ge1\}$ such that $u_{i_j}-v_{i_j}\to0$ as $j\to\infty$,
which contradicts to the fact $y_i-x_i=f(u_i)-f(v_i)>1$.
Thus, there exists an integer $i_0=i_0\big(G_2,f(\cdot)\big)\ge1$ such that, when $i\ge i_0$, it holds by $f(\cdot)\in\mathcal{L}_0$ and (\ref{409}) that
\begin{eqnarray*}
&\overline{G_{02}}\big(x_{i}\big)=\overline{G_{02}}\big(f(u_{i}-b)\big)=\overline{G_{02}}\big(f(u_{i}-2^{-1}b)\big)=\overline{G_{02}}\big(f(u_{i})\big)
=\overline{G_{02}}(y_{i})>\overline{G_{02}}(y_i+0)=\overline{G_{02}}\big(f(u_{i})+0\big).
\end{eqnarray*}
Then we construct the distribution $G_{03}$ as follows: in $\overline{G_{02}}(x)$, we link $\overline{G}_{02}\big(f(u_{i}-2^{-1}b)\big)$ and $\overline{G}_{02}\big(f(u_{i}+0)\big)=\overline{G_{02}}(y_{i}+0)$ with line segment for each $i\ge1$. In this way, we get a continuous distribution $G_{03}$ on $[1,\infty)$
such that $\overline{G_{03}}(x)\asymp\overline{G_{02}}(x)\asymp\overline{G_{01}}(x)$, $c=\int_1^\infty \overline{G_{03}}(y)dy<\infty$. Therefore,
$$\lim_{i\to\infty}\overline{G_{03}}\big(f(u_{i}-b)\big)\overline{G_{03}}^{-1}\big(f(u_{i})\big)=\lim\overline{G_{01}}(x_i)\overline{G_{01}}^{-1}(y_i)=a>1,$$
which implies that $G_{03}\notin\mathcal{L}$.

Fourthly, by $G_{03}$, we construct a distribution $G$ on $[1,\infty)$ {of r.v. $Y_1$} such that
\begin{eqnarray*}
&\overline{G}(x)=\textbf{\emph{1}}_{(-\infty,1]}(x)+c^{-1}\int_x^\infty \overline{G_{03}}(y)dy\textbf{\emph{1}}_{(1,\infty)}(x),\ \ \ \ \ \ x\in(-\infty,\infty).
\end{eqnarray*}
Then, the function $I(\cdot)$ corresponding to $G$ does not belong to the class $\mathcal{L}_0$. In fact, for each $$a^{-\alpha^{-1}}<t<1$$
and the above $b>0$, by the continuity of $G$, we have
\begin{eqnarray*}
&\limsup\frac{I(xt)}{I(x)}=\limsup\frac{\int_1^{f(xt)}y^\alpha G(dy)}{\int_1^{f(x)}y^\alpha G(dy)}
=\limsup\frac{f^\alpha(xt)f'(xt)\overline{G_{03}}(f(xt))}{f^\alpha(x)f'(x)\overline{G_{03}}(f(x))}
\ge \limsup\frac{t^\alpha\overline{G_{03}}(f(x-b))}{\overline{G_{03}}(f(x))}\ge t^\alpha a>1,
\end{eqnarray*}
which leads to $I(\cdot)\notin\mathcal{R}^0_0$.

Finally, for the distribution $F$, we have by Theorem \ref{thm302} that
\begin{eqnarray*}
&I(x)=c^{-1}\int_1^{f(x)}y^\alpha\overline{G_{03}}(y)dy\asymp\ln x
\ \ \ \ \ \ \ \text{and}\ \ \ \ \ \ \ \psi(x;n)\asymp\ln^nx\overline{F}(x).
\end{eqnarray*}
Hence, the distribution of $\max_{1\le k \le n}S_k$ does not belong to the class $\mathcal{L}$, thus does not belong to the class $\mathcal{R}$, but still belongs to the class $\mathcal{D}$.
$\hspace{\fill}\Box$
\end{exam}

\end{document}